\documentclass[10pt]{amsart}
\usepackage{amssymb}
\usepackage{verbatim}
\usepackage{xspace}
\usepackage{ulem}

\newcommand{\R}{\mathbb{R}}

\newcommand{\di}{\,d}

\newcommand{\SB}{\mathbb{S}}

\newcommand{\M}{\mathcal{M}}
\newcommand{\PP}{\mathcal{P}} \newcommand{\LL}{\mathcal{L}}

\renewcommand{\H}{\mathcal{H}} 
 
 \renewcommand{\Cap}{\gamma}
 \newcommand{\tn}{\textnormal}
\newcommand{\po}{\partial} 
 
\newcommand{\ve}{\varepsilon} 
 
 \newcommand{\spt}{{\text{\sl spt\,}}}

\newcommand{\DM}{\mathcal D\mathcal M} 
 \renewcommand{\div}{\textnormal{div }}

 \newcommand{\F}{\mathcal{ F}}
\theoremstyle{plain}
\newtheorem{theorem}{Theorem}[page]

\newtheorem{lemma}[theorem]{Lemma}

\numberwithin{equation}{page}

\usepackage{amsmath}
\usepackage{amsthm}
\usepackage{amsfonts}
\usepackage{amssymb}
\usepackage{verbatim}
\usepackage{hhline}
\usepackage[mathscr]{eucal}
\usepackage{enumerate}
\newtheorem{definition}[theorem]{Definition}

\newtheorem{proposition}[theorem]{Proposition}
\newtheorem{example}[theorem]{Example}
\newtheorem{corollary}[theorem]{Corollary}
\newtheorem{cor}[theorem]{Corollary}
\newtheorem{remark}[theorem]{Remark}

\newcommand{\charfn}[1]{\raisebox{1.2pt}{$\chi$}
           \hspace{-1pt}\raisebox{-3pt}{{$\scriptscriptstyle#1$}}}

\newcommand{\abs}[1]{\left\vert{#1}\right\vert}
\newcommand{\rn}{\mathbb{R}^{N}}
\newcommand{\ri}{\mathbb{R}}
\newcommand{\rii}{\mathbb{R}^{2}}
\newcommand{\norm}[1]{\left\Vert#1\right\Vert}

\renewcommand{\div}{\text{\sl div\,}}

\newcommand{\dist}{{\boldsymbol d}}
\newcommand{\FF}{{\boldsymbol F}}
\newcommand{\nnu}{{\boldsymbol \nu}}
\newcommand{\pp}{{\boldsymbol p}}

\def\medcup{\mathop{\textstyle\bigcup}\limits}
\def\medcap{\mathop{\textstyle\bigcap}\limits}

\def\rightangle{\vcenter{\hsize5.5pt
    \hbox to5.5pt{\vrule height7pt\hfill}
    \hrule}}
\def\rtangle{\mathrel{\rightangle}}
\def\intave#1{\int_{#1}\hbox{\llap{$\raise2.3pt\hbox{\vrule
height.9pt width7pt}\phantom{\scriptstyle{#1}}\mkern-2mu$}}}

\def\intav#1{\mathchoice
          {\mathop{\vrule width 9pt height 3 pt depth -2.6pt
                  \kern -9pt \intop}\nolimits_{\kern -6pt#1}}%
          {\mathop{\vrule width 5pt height 3 pt depth -2.6pt
                  \kern -6pt \intop}\nolimits_{#1}}%
          {\mathop{\vrule width 5pt height 3 pt depth -2.6pt
                  \kern -6pt \intop}\nolimits_{#1}}%
          {\mathop{\vrule width 5pt height 3 pt depth -2.6pt
                  \kern -6pt \intop}\nolimits_{#1}}}
\def\intav#1{\vint_{#1}}
\allowdisplaybreaks[1]
\allowdisplaybreaks[1]
\usepackage[displaymath,textmath,sections,graphics,floats]{preview}
     \PreviewEnvironment{enumerate}

\begin{document}
\title[Gauss-Green theorem, sets of finite perimeter, and balance
laws]{Gauss-Green Theorem for Weakly Differentiable Vector Fields,
Sets of Finite Perimeter, and Balance Laws}
%
\author{Gui-Qiang Chen \qquad Monica Torres  \qquad William P. Ziemer}
\address[G.-Q. Chen]{
 Department of Mathematics, Northwestern University,
 2033 Sheridan Road, Evanston, IL 60208-2730, USA.
 http://www.math.northwestern.edu/\~{}gqchen}
\email{gqchen@math.northwestern.edu}
\address[M. Torres]{Department of Mathematics\\
 Purdue University,
 150 N. University Street, West Layayette, IN 47907-2067, USA.
 http://www.math.purdue.edu/\~{}torres}
\email{torres@math.purdue.edu}
\address[W. Ziemer]{Department of Mathematics\\
 Indiana University,
 Rawles Hall, Bloomington, IN 47405, USA.
 http://www.indiand.edu/\~{}ziemer}
 \email{ziemer@indiana.edu}

\thanks{Submitted on September 3, 2007}

\begin{abstract}
We analyze a class of weakly differentiable vector fields \(\FF
\colon \rn \to \rn\) with the property that \(\FF\in L^{\infty}\)
and \(\div \FF\) is a Radon measure.  These fields are called
bounded divergence-measure fields. The primary focus of our
investigation is to introduce a suitable notion of the normal trace
of any divergence-measure field $\FF$ over the boundary of an
arbitrary set of finite perimeter, which ensures the validity of the
Gauss-Green theorem. To achieve this, we first develop an
alternative way to establish the Gauss-Green theorem for any smooth
bounded set with \(\FF\in L^{\infty }\).  Then we establish a
fundamental approximation theorem which states that, given a Radon
measure $\mu$ that is absolutely continuous with respect to
$\mathcal{H}^{N-1}$ on $\rn$, any set of finite perimeter can be
approximated by a family of sets with smooth boundary essentially
from the measure-theoretic interior of the set with respect to the
measure $\|\mu\|$.
We employ this approximation theorem to derive the normal trace of
$\FF$ on the boundary of any set of finite perimeter, \(E\), as the
limit of the normal traces of $\FF$ on the boundaries of the
approximate sets with smooth boundary, so that the Gauss-Green
theorem for $\FF$ holds on \(E\). With these results, we analyze the
Cauchy fluxes that are bounded by a Radon measure over any oriented
surface (i.e. an $(N-1)$-dimensional surface that is a part of the
boundary of a set of finite perimeter) and thereby develop a general
mathematical formulation of the physical principle of balance law
through the Cauchy flux.  Finally, we apply this framework to the
derivation of systems of balance laws with measure-valued source
terms from the formulation of balance law. This framework also
allows the recovery of Cauchy entropy fluxes through the Lax entropy
inequality for entropy solutions of hyperbolic conservation laws.
\end{abstract}

\maketitle

\section{Introduction}

\medskip
In this paper we analyze a class of weakly differentiable vector
fields \(\FF\colon \rn \to \rn \) with the property that \(\FF\in
L^{\infty }\) and \(\div \FF\) is a Radon measure $\mu$ with finite
total variation (i.e. a totally finite signed measure).
These fields are called {\it bounded divergence-measure fields}, and
the class is denoted by \(\mathcal{DM}^{\infty}\).  The primary
focus of our investigation is to introduce a suitable notion of the
normal trace of any divergence-measure field over the boundary of an
arbitrary set of finite perimeter to obtain a general version of the
Gauss-Green theorem.  Clearly, this investigation is closely related
to the theory of $BV$ functions in \(\rn \); in fact, it would be
completely subsumed by the $BV$ theory if the fields were of the
form \(\FF=(F_{1},F_{2}, \dots, F_{N}\)) with each \(F_{k}\in BV(\rn
) \), since then $\div \FF=\sum_{k=1}^{N}\frac{\partial
F_{k}}{\partial x_{k}}$ is a Radon measure $\mu$ with finite total
variation (cf. \cite{Ziemer2}).
However, in general,
the condition \(\div \FF=\mu \) allows for cancellation, which thus
makes the problem more difficult and accordingly more important for
applications (see \S \ref{sec:cauchy}--\S\ref{sec:conservation}).
For the Gauss-Green theorem in the $BV$ setting, we refer to
Burago-Maz'ja \cite{BM2}, Volpert \cite{Volpert}, and the references
therein. The Gauss-Green theorem for Lipschitz vector fields over
sets of finite perimeter was first obtained by DeGiorgi
\cite{Georgi1, Georgi2} and Federer \cite{Federer1,Federer2}. Also
see Evans-Gariepy \cite{E1}, Lin-Wang \cite{LinYang}, and Simon
\cite{Simonbook}.

Some earlier efforts were made on generalizing the Gauss-Green
 theorem for some special situations of divergence-measure fields,
and relevant results can be found in Anzellotti \cite{Anz} for an
abstract formulation when $\FF\in L^\infty$ over a set with $C^1$
boundary and Ziemer \cite{Ziemer1} for a related problem for
$\div\,\FF\in L^\infty$; also see
\cite{afp,Ba1,BF1,jurkat1,jurkat2,jurkat5,Nonnenmacher,pf1,pf2,pf3,Rodrigues}.
In Chen-Frid \cite{CF1,CF-CMP}, an explicit way to formulate the
suitable normal traces over Lipschitz deformable surfaces was first
observed for $\FF\in \DM^\infty$. In particular, it was proved in
\cite{CF1,CF-CMP} that the normal trace over a Lipschitz deformable
surface, oriented by the unit normal vector $\nu$, is determined
completely by the neighborhood information from the positive side of
the surface pointed by $\nu$ and is independent of the information
from the other side. This is the primary motivation for our further
investigation on divergence-measure fields. Chen-Torres
\cite{ChenTorres} were the first to obtain the normal trace for any
bounded divergence-measure field over a set of finite perimeter,
$E$, and the corresponding Gauss-Green theorem. One of the main
results in this paper is to obtain this normal trace as the limit of
the normal traces over the smooth boundaries that approximate the
reduced boundary $\partial^{*}E $ of $E$.  In particular, the normal
trace is determined completely by the neighborhood information
essentially from the measure-theoretic interior of the set (see
Theorem \ref{main}), so that the Gauss-Green theorem holds for any
set of finite perimeter.

We recall a very general approach, initiated by Fuglede
\cite{fuglede}, in which the following result was established: If
$\FF\in L^{p}(\R^{N};\R^{N})$, $1\le p\le\infty$, is a vector field
with ${\div}  \FF=\mu $, where $\mu $ is a signed Radon measure,
then
\begin{equation}\label{one}
  \int_{E}\div\FF:=\mu (E)=-\int_{\partial ^{*}E}\FF(y)\cdot\nu
(y)\,d\H^{N-1}(y)
\end{equation}
for ``almost all'' sets of finite perimeter, $E$, where
$\mathcal{H}^{N-1}$ is the $(N-1)$-dimensional Hausdorff measure.
The term ``almost all'' is expressed in terms of a condition that
resembles ``extremal length'', a concept used in complex analysis
and potential theory (cf. \cite{zcc,zel,zexl,He}). One way of
summarizing our work in this paper is to say that we wish to extend
Fuglede's result so that \eqref{one} holds for {\it every} set \(E\)
of finite perimeter. Of course, this requires a suitable notion of
the normal trace of \(\FF\) on \(\partial^{*}E \). This is really
the crux of the problem as \(\FF\), being only measurable, cannot be
re-defined on an arbitrary set of dimension \(N-1\). To achieve our
goal, we first establish a fundamental approximation theorem which
states that, given a Radon measure $\mu$ on $\rn$ such that
$\mu<<\mathcal{H}^{N-1}$, any set of finite perimeter can be
approximated by a family of sets with smooth boundary essentially
from the measure-theoretic interior of the set with respect to the
measure $\|\mu\|$ (e.g. $\mu=\div \FF$).
Then we employ this approximation theorem to derive the normal trace
of $\FF$ on the boundary of any set of finite perimeter as the limit
of the normal traces of $\FF$ on the smooth boundaries of the
approximate sets and establish the Gauss-Green theorem for $\FF$
which holds for an arbitrary set of finite perimeter.

With these results on divergence-measure fields and sets of finite
perimeter, we analyze the Cauchy flux that is bounded by a
nonnegative Radon measure $\sigma$ over an oriented surface (i.e. an
$(N-1)$-dimensional surface that is a part of the boundary of a set
of finite perimeter) and develop a general mathematical formulation
of the physical principle of balance law through the Cauchy flux. In
the classical setting of the physical principle of balance law,
Cauchy \cite{Cauchy1, Cauchy2} first discovered that the flux
density is necessarily a linear function of the interior normal
(equivalently, the exterior normal) under the assumption that the
flux density through a surface depends on the surface solely through
the normal at that point. It was shown in Noll \cite{Noll} that
Cauchy's assumption follows from the balance law. Ziemer
\cite{Ziemer1} provided a first formulation of the balance law for
the flux function $\FF\in L^\infty$ with $\div \FF\in L^\infty$ at
the level of generality with sets of finite perimeter. Also see
Dafermos \cite{Da} and Gurtin-Martins \cite{gw1,gw2}. One of the new
features in our formulation is to allow the presence of exceptional
surfaces, ``shock waves'', across which the Cauchy flux has a jump.
When the Radon measure $\sigma$ reduces to the $N$-dimensional
Lebesgue measure $\mathcal{L}^N$, the formulation reduces to
Ziemer's formulation in \cite{Ziemer1}, which shows its consistency
with the classical setting. We first show that, for a Cauchy flux
$\mathcal{F}$ bounded by a measure $\sigma$, there exists a bounded
divergence-measure field $\FF: \R^N\to \R^N$, defined
$\mathcal{L}^N$-a.e., such that
$$
\mathcal{F}(S)=-\int_S \FF(y)\cdot \nu(y) \,\,d\mathcal{H}^{N-1}(y)
$$
for almost any oriented surface $S$, oriented by the interior unit
normal $\nu$. Then we employ our results on divergence-measure
fields to recover the values of the Cauchy flux on the exceptional
surfaces directly via the vector field $\FF$. The value as the normal
trace of $\FF$ on the exceptional surface is the unique limit of the
normal traces of $\FF$ on the non-exceptional surfaces essentially
from the positive side of the exceptional surface pointed by $\nu$.
Finally, we apply this general framework to the derivation of
systems of balance laws with measure-valued source terms from the
mathematical formulation of balance law. We also apply the framework
to the recovery of Cauchy entropy fluxes through the Lax entropy
inequality for entropy solutions of hyperbolic conservation laws by
capturing entropy dissipation.

\medskip
We observe the recent important work by Bourgain-Brezis \cite{BB}
and De Pauw-Pfeffer \cite{pf2} for the following problem with
different view of point: Find a continuous vector field to the
divergence-measure equation:
\begin{equation}\label{DPP}
\div \FF =\mu \qquad \text{in} \,\, \Omega,
\end{equation}
for a given Radon measure $\mu$. In the case $d\mu=f \, dx$ where
$f\in L^n_{loc}(\Omega)$, the existence of a solution $\FF$ to
\eqref{DPP} follows form the closed-range or Hahn-Banach theorem as
shown in \cite{BB}. It is proved in \cite{pf2} that equation
\eqref{DPP} has a continuous weak solution if and only if $\mu$ is a
strong charge, i.e., given $\varepsilon>0$ and a compact set
$K\subset \Omega$, there is $\theta>0$ such that
$$
\int_\Omega \phi\, d\mu \le \varepsilon \|\nabla\phi\|_{L^1}
+\theta\|\phi\|_{L^1}
$$
for any smooth function $\phi$ compactly supported on $K$.

\bigskip
The organization of this paper is as follows. In \S 2, we first
recall some properties of Radon measures, sets of finite perimeter,
and related $BV$ functions, and then we introduce the notion of
 an oriented surface and develop some basic properties of
divergence-measure fields. In \S 3, we develop an alternative way to
obtain the Gauss-Green formula for a bounded divergence-measure
field over any smooth boundary by a technique, which motivates our
further development for the general case. In \S 4, we establish a
fundamental approximation theorem which states that, given a Radon
measure $\mu$ on $\R^N$ such that $\mu << \mathcal{H}^{N-1}$, any
set of finite perimeter can be approximated by a sequence of sets
with smooth boundary essentially from the interior of the set with
respect to the measure $\|\mu\|$. In \S 5, we introduce the normal
trace of a divergence-measure field $\FF$ on the boundary $\partial
E$ of any set of finite perimeter as the limit of the normal traces
of $\FF$ on the smooth surfaces that approximate $\partial E$
essentially from the measure-theoretic interior of $E$ with respect
to the measure $\|\div \FF\|$, constructed in \S 4, and then we
establish the corresponding Gauss-Green theorem. In \S6--\S7, we
further analyze properties of divergence-measure fields, especially
showing the representation of the divergence-measures of jump sets
via the normal traces and the consistency of our normal traces with
the classical traces (i.e. values) when the vector field is
continuous. In \S 8, we first show that, if the set of finite
perimeter, $E$, satisfies \eqref{8.1} (which is similar to Lewis's
``uniformly fat'' condition in potential theory \cite{Lewis}), there
exists a one-sided approximation to $E$, and we then show that an
open set of finite perimeter is an extension domain for any bounded
divergence-measure field. In \S 9, we first introduce a class of
Cauchy fluxes that allow the presence of these exceptional surfaces
or ``shock waves'', and we then prove that such a Cauchy flux
induces a bounded divergence-measure (vector) field $\FF$ so that
the Cauchy flux over {\it every} oriented surface with finite
perimeter can be recovered through $\FF$ via the normal trace over
the oriented surface.

In \S 10, we apply the results established in \S3--\S 9 to the
mathematical formulation of the physical principle of balance law
and the rigorous derivation of systems of balance laws with
measure-valued source terms from that formulation. Finally, in \S
11, we apply our results to the recovery of Cauchy entropy flux
through the Lax entropy inequality for entropy solutions of
hyperbolic conservation laws by capturing entropy dissipation.

\section{Radon Measures, Sets of finite perimeter, and
Divergence-Measure Fields}

In this section we first recall some properties of Radon measures,
sets of finite perimeter, and related $BV$ functions (also cf.
\cite{afp,E1,Fe,G1,Ziemer2}). We then introduce the notion of oriented
surfaces and develop some basic properties of divergence-measure
fields.  For the sake of completeness, we start with some basic
notions and definitions. First, denote by $\mathcal{H}^{M}$ the
$M$-dimensional Hausdorff measure in $\R^{N}$ for $M\le N$, and by
$\mathcal{L}^{N}$ the Lebesgue measure in $\R^{N}$ (recall that
$\mathcal{L}^{N}=\mathcal{H}^{N}$). For any
$\mathcal{L}^{N}$-measurable set $E \subset \R^{N}$, we denote $|E|$
as the $\mathcal{L}^{N}$-Lebesgue measure of the set $E$ and $\partial
E$ as its topological boundary. Also, we denote $B(x,r)$ as the closed
ball of radius $r$ and center at $x$. The symmetric difference of sets
is denoted by
$$
A \Delta B:=(A\setminus B)\cup (B\setminus A).
$$
Let $\Omega\subset\rn$ be open. We denote by \(E \Subset\Omega \)
that the closure of  \(E\) is compact and contained in \(\Omega \).
Let \(C_{c}(\Omega)\) be the space of compactly supported continuous
functions on \(\Omega\) with \(\norm{ \varphi }_{0;\Omega }:=
\sup\{|\varphi (y)|\,:\, y\in \Omega\}.\) A Radon measure \(\mu \)
in $\Omega$ is a regular Borel measure whose total variation on each
compact set \(K \Subset \Omega\) is finite, i.e. \(\|\mu\|
(K)<\infty \). The space of Radon measures supported on an open set
\(\Omega \) is denoted by $\mathcal{M}(\Omega)$. Any Radon measure
$\mu$ can be decomposed into the difference of two positive Radon
measures $ \mu =\mu^+-\mu^-$; the total variation of $\mu$ is
$\norm{\mu}=\mu^++\mu^-$. Equivalently, if \(\mu\) is a signed Radon
measure on \(\Omega\), the total variation of \(\mu\) on any bounded
open set $B\subset\Omega$ is equal to
\begin{equation}\label{totalvariation}
\norm{ \mu }(B) = \sup\Big\{\int_{\Omega}\varphi \di \mu :
\varphi\in C_c(B),\;\norm{ \varphi }_{0;\Omega
}\le1\Big\}=\sup\Big\{\sum_{i=0}^{\infty}\abs{\mu (B_i)}\Big\}
\end{equation}
where the second supremum is taken over all pairwise disjoint Borel
sets \(B_{i}\) with $ B=\medcup_{i=1}^{\infty}B_i$. Since the space
of Radon measures can be identified with the dual of
\(C_{c}(\Omega)\), we may consider a Radon measure \(\mu\) as a
linear functional on \(C_{c}(\Omega)\), written as
\begin{equation}
  \label{dual}
  \mu (\varphi ):= \int_{\Omega}\varphi \di \mu
  \qquad\text{for each \(\varphi \in C_{c}(\Omega).\)}
\end{equation}
We recall the familiar weak*-topology on \(\mathcal{M}(\Omega)\)
which, when restricted to a sequence $\{\mu_k\}$, yields
\[
\mu_{k}\, \stackrel{*}{\rightharpoonup} \, \mu \qquad
\mbox{in}\,\,\mathcal{M}(\Omega),
\]
that is, \(\mu_{k} \) converges to \(\mu\) in the weak* topology if
and only if
\begin{equation}
\label{wkconvdef} \mu_{k} (\varphi )\to \mu (\varphi
)\qquad\text{for each \(\varphi \in C_{c}(\Omega)\)}.
\end{equation}

The space \(L^{p}(\Omega,\mu )\), \(1\le p\le \infty \), denotes all
the functions \(f\) with the property that \(\abs{f}^{p}\) is
\(\mu\)-integrable. The conjugate of \(p\) is \(q:= p/(p-1) \). The
$L^{p}$ norm of $f$ on a set $E$ with integration taken with respect
to a measure $\mu$ is denoted by $\norm{f}_{p;E,\mu}$.  In the event
$\mu$ is Lebesgue measure, we will simply write $\norm{f}_{p;E}$\,.

\begin{theorem}[Uniform boundedness principle]\label{unbdprin}
Let \(X\) be a Banach space. If \(T_{k}\) is a sequence of linear
functionals on $X$ which converge weak* to \(T\). Then
\[
\limsup_{k \to  \infty }\norm {T_{k}}< \infty.
\]
\end{theorem}

This theorem implies the following corresponding result for Radon
measures.
\begin{corollary}\label{unbMeasures}
Let \(\mu_{k} \) be a sequence of Radon measures that converge to
\(\mu\) in the weak* topology. Then
\[
\limsup_{k \to  \infty }\norm {\mu _{k}}< \infty.
\]
\end{corollary}
Next, we quote a familiar result concerning weak*-convergence.
\begin{lemma}\label{wkstconvergence}
Let \(\mu\) be a Radon measure on \(\Omega\) and let \(\mu _{k} \)
be a sequence of Radon measures converging weak* to \(\mu  \). Then\\
{\rm (i)} If $A\subset\Omega$ is any open set and $\mu_k$ are
positive Radon measures,
$$
\mu(A)\leq \liminf_{k \to \infty} \mu_k(A);
$$
{\rm (ii)} If $K\subset\Omega$ is any compact set and $\mu_k$ are
positive Radon measures,
$$
\mu(K)\geq \limsup_{k \to \infty} \mu_k(K);
$$
{\rm (iii)} If $\norm{\mu_k}\, \stackrel{*}{\rightharpoonup}
\,\sigma$,  then $\norm{\mu} \leq \sigma$. In addition, if $E$
satisfies $\sigma(\po E)=0$, then
$$
\mu(E) =\lim_{k \to \infty} \mu_k(E).
$$
More generally, if \(f \) is a bounded Borel function with compact
support in $\Omega$ such that the set of its discontinuity points is
$\sigma$-negligible, then
\[
\lim_{k\to\infty}\int_{\Omega}f\di \mu_{k}=\int_{\Omega}f\di \mu.
\]
\end{lemma}

\begin{definition}
\label{def1} For every $\alpha \in[0,1] $ and every
$\mathcal{L}^{N}$-measurable set $E \subset \rn$, define
\begin{equation}
\label{densityone} E^{\alpha}:= \{y \in \rn  \, : \, D(E,y)=\alpha
\},
\end{equation}
where
\begin{equation}
D(E,y):= \lim_{r \rightarrow 0} \frac {\vert E \cap B(y,r) \vert}
{\vert  B(y,r) \vert} \label{densitytwo}.
\end{equation}
Then $E^{\alpha}$ is the set of all {\it points with density
$\alpha$}. We define the {\it {measure-theoretic boundary}} of $E$,
$\partial^{m}E$, as
\begin{equation}
\partial^{m}E:=\rn \setminus (E^{0}\cup E^{1}).
\end{equation}
\end{definition}

\begin{definition} \label{def2} A function \(f\colon\Omega\to\ri \) is
called a function of bounded variation if each partial derivative of
\(f \) is a totally finite signed Radon measure. Notationally, we
write \(f\in BV (\Omega). \) Let $E\Subset\Omega$ be an
$\mathcal{L}^{N}$-measurable subset. We say that $E$ is a {\it set
of finite perimeter} if $\mathcal{H}^ {N-1}(\po^{m}E) < \infty$.
Equivalently, \(E \) is of finite perimeter if $\charfn{E}\in BV
(\Omega) $. Consequently, if $E$ is a set of finite perimeter, then
$\nabla \charfn{E}$ is a (vector-valued) Radon  measure whose total
variation is denoted by $\norm{\nabla\charfn{E}}$.
\end{definition}

\begin{definition}
\label{reducedboundary} Let $E\Subset\Omega$ be a set of finite
perimeter. The {\it {reduced boundary}} of $E$, denoted as
$\partial^{*}E$, is the set of all points $y\in\Omega$ such that
\begin{enumerate}[\rm (i)]
\item \(\norm{\nabla \charfn{E} } (B(y,r))>0
\) for all \(r>0 \);
\item The limit \(\nu_{_{E}} (y):=\lim_{r\to0}\frac{\nabla
\charfn{E}(B(y,r))}{\norm{\nabla \charfn{E}} (B(y,r))}\)
exists.
\end{enumerate}
Then, for \(\H^{N-1} \)-a.e. \(y \in \partial^{*}E \),
\begin{equation*}
\lim_{r \to 0} \frac{\norm{\nabla\chi_E}(B(y,r))}
{\alpha(N-1)r^{N-1}}=1,
\end{equation*}
where $\alpha(N-1)$ is the Lebesgue measure of the unit ball in
$\R^{N-1}$, and the generalized gradient of \(\charfn{E}\) enjoys
the following basic relationship with \(\H^{N-1}\):
\begin{equation}
  \label{normHaus}
\Vert\nabla \charfn{E}\Vert =\H^{N-1}\rtangle \partial^{*}E.
\end{equation}
The unit vector, $\nnu_{_{E}}(y)$, is called the {\it
{measure-theoretic interior unit normal to $E$ at $y$}} (we
sometimes write \(\nnu \) instead of \(\nnu _{E}\) for notational
simplicity). Also, we recall that the reduced boundary,
\(\partial^{*}E \), is an \((N-1)\)-rectifiable set which implies
that there exists a countable family of \(C^{1}\)-manifolds
\(M_{k}\) of dimension \(N-1\) and a set \(\mathcal{N} \) of
$\H^{N-1}$ measure zero
such that
\begin{equation}
  \label{almostsmooth}
  \partial^{*}E \subset \big(\medcup_{k=1}^{\infty }M_{k}\big)\medcup
\mathcal{N}.
\end{equation}
\end{definition}
In view of the following, we see that \(\nnu=\nnu _{_{E}} \) is aptly
named because \( \nnu  \)  is the interior unit normal to \( E \)
provided that \( E \) (in the limit and in measure) lies in the
appropriate half-space determined by the hyperplane orthogonal to \(
\nnu  \); that is, \(\nnu \) is the interior unit normal to $E$ at $x$
provided that
$$
D(\{y:(y-x)\cdot\nnu>0,y\notin E\} \cup \{y:(y-x)\cdot\nnu<0,y\in E\},
\, y)=0.
$$

We will refer to the sets $E^{0}$ and $E^{1}$ as the {\it
measure-theoretic exterior and interior} of $E$. We note that, in
general, the sets $E^{0}$ and $E^{1}$ do not coincide with the
topological exterior and interior of the set $E$. The sets $E^{0}$
and $E^{1}$ also motivate the definition of measure-theoretic
boundary. Indeed, for any set \( E\subset\rn \), the definitions
imply that \( \rn=E^{1}\cup\partial ^{*}E \cup E^{0}\cup \mathcal{N}
\) where \(\H^{N-1} (\mathcal{N})=0\). If we define a set \( E \) to
be ``open'' if $E$ is both measurable and $D(E,x)=1$ for all \( x\in
E \), then this concept of openness defines a topology, called the
{\it density topology}. It is an interesting exercise to prove that
the open sets are closed under arbitrary unions; the crux of the
problem is to prove that the arbitrary union is, in fact,
measurable. This topology is significant
because it is the smallest topology (the one
with the smallest number of open sets) for which the approximately
continuous functions are continuous \cite{gnn}.

\begin{remark}
\label{rem00}
If $E$ is a set of finite perimeter, then clearly
\begin{equation}\label{2.2a1}
\partial^{*}E \subset  E^{\frac{1}{2}} \subset \partial^{m}E,
\quad \mathcal{H}^{N-1}(\partial^m E\setminus \po^{*}E)=0.
\end{equation}
\end{remark}

The following result, which is easily verified (although tedious),
will be needed in the sequel.

\begin{lemma}\label{tedious}
  If \(A,B \Subset \Omega\) are sets of finite perimeter, then
\[
\partial^{m} (A\cap B)=\big(\partial^{m} A\cap B\big)\cup \big(A\cap
\partial^{m} B\big)\cup \big(\partial^{m}A\cap \partial^{m}B\big).
\]
\end{lemma}

\begin{definition}\label{mollifiers}
Let \(\rho\in C_{c}^{\infty }(\rn ) \) be a standard symmetric {\it
mollifying kernel}; that is, \(\rho \) is a nonnegative function
with support in the unit ball and satisfies \(\norm{\rho }_{1;\rn}=1
\). With \(u\in L^{1}(\rn ) \), we set \(u_{\varepsilon
}:=u*\rho_{\varepsilon } \), where the sequence \( \rho_{\varepsilon
}(y):=\frac{1}{\varepsilon^N}\rho(\frac{y}{\varepsilon}) \) forms a
mollifier.
\end{definition}

Recall the following properties of mollification (cf.
\cite{Ziemer2}):
\begin{lemma}\label{mollifier-properties}
\begin{enumerate}[\rm(i)]
\item If $u \in
L^{1}_{\mbox{\scriptsize loc}}(\rn)$, {\it then, for every\/}
$\varepsilon > 0$, $u_{\varepsilon} \in C^{\infty}(\rn)$ {\it and\/}
$D^{\alpha}(\rho_{\varepsilon}\ast u) =
(D^{\alpha}\rho_{\varepsilon})\ast u$ {\it for each multi-index\/}
$\alpha$;
\item With \(k:=1/\varepsilon_k\) and \(\varepsilon_k \to0  \),
$u_{k}(x) \rightarrow u(x)$ {\it whenever\/} $x$ {\it is a Lebesgue
point of\/} $u$. {\it In particular, if\/} $u$ {\it is continuous,
then\/} $u_{\varepsilon}$ {\it converges uniformly to\/} $u$ {\it on
compact subsets of\/} $\rn$.
\end{enumerate}
\end{lemma}

When \(u \) is taken as \(\charfn{E}\) for  a set of finite
perimeter, \(E\), this result can be considerably strengthened.

\begin{lemma}\label{mollifiers2}
If $u_k$ is the mollification of $\charfn{E}$ for a set of finite
perimeter, \(E \), then the following hold:
\begin{enumerate}[\rm (i)]
\item \( u_{k}\in C^{\infty }(\rn)\);
\item There is a set \(\mathcal{N}  \) with
      \(\H^{N-1}(\mathcal{N})=0 \) and a function \(u_{_{E }} \in BV\) such that,
      for all \(y\notin \mathcal{N}\), \(u_{k}(y)\to u_{_{E}}(y)\)
      as $k\to \infty$ and
\[
u_{_{E}}(y)=
\begin{cases}
1&\text{$y\in E^{1}$,}\\
\frac{1}{2}&\text{$y\in\partial ^{*}E$,}\\
0&\text{$y\in E^{0}$;}
\end{cases}
\]
\item $\nabla u_{k} \, \stackrel{*}{\rightharpoonup}
\, \nabla u_{_{E }} \quad\text{in} \,\, \mathcal{M}(\R^N)$;
\item \(\norm{\nabla u_{k}}(U)\to\norm{\nabla u_{_{E }}}(U)\) as
$k\to\infty$, \,\text{for any open set \(U\) with \(\norm{\nabla
u_{_{E }}} (\partial U)=0\)};
\item  \(\nabla \charfn{E} =\nabla u_{_{E }}\).
\end{enumerate}
\end{lemma}

\begin{proof}
  Only (iii) requires a proof, since (i), (ii), and (iv) are the results
  from the standard $BV$ theory and (v) is immediate from the
  definitions and the fact that \(u_{_{E }} =\charfn{E}\) almost
  everywhere. As for (iii), since \(u_{k} \to u_{_{E }} \) in
  \(L^{1}(\rn )\), then \(u_{k} \to u_{_{E }} \) when considered as
  distributions, which implies that \(\nabla u_{k} \to \nabla u_{_{E
    }} \) as distributions and consequently as measures since \(\nabla
  u_{k}, \nabla u_{_{E }}\in \mathcal{M}(\rn)\).
\end{proof}

\begin{remark}\label{precise} More generally, functions in the spaces
  \(BV(\rn )\) and \(W^{1,p}(\rn )\), \(1\le p \le \infty \), have
  precise representatives; that is, if \(u\in BV(\rn )\), then there
  is a function \(u^{*}\in BV(\rn ) \) such that \(u\) and \(u^{*} \)
  are equal a.e. and that the mollification sequence of \(u\), \(u_{k} \),
  converges to \(u^{*} \) at all points except those that belong to an
  exceptional set \(E\) with \(\H^{N-1} (E)=0\). However, this is not
  the same as saying that \(u\) has a Lebesgue point, which is
  slightly stronger.
  A similar statement is true for functions in the Sobolev space
\(W^{1,p}(\rn )\), \(1<p\le \infty \), except that the exceptional
set \(E\) has \(\gamma_{p} \)-capacity zero, see Definition {\rm
\ref{def:capacity}} below. As we will see, the
\(\gamma_{1}\)-capacity vanishes precisely on sets of \(\H^{N-1} \)
measure zero. Thus, we can say that functions in the spaces \(BV\)
and \(W^{1,p}\) have precise representatives that are defined,
respectively, \(\gamma _{1}\) and \(\gamma_{p} \)  almost
everywhere.

\end{remark} The next result affirms the notion that the mollification is
generally a norm reducing operation.

\begin{lemma} \label{decreasing}
Let \(E\) be a set of finite perimeter and let \(u_{k}\) denote the
mollification of \(\charfn{E} \). Then
\[
\norm{\nabla u_{k}}_{1}\le\norm{\nabla \charfn{E} }.
\]
\end{lemma}

\begin{proof}
For any $f \in BV(\rn)$, consider the convolutions $f_{\varepsilon
}(y)=\int_{\rn }\rho_{\varepsilon } (y-x)f(x)\di x$. Using $\nabla
f_{\varepsilon }=\rho_{\varepsilon }*\nabla f$ and \(f_{\varepsilon
}\in C^{\infty }(\rn) \), we obtain
\[
\nabla f_{\varepsilon }(y)=\int_{\rn }\rho_{\varepsilon }(y-x)\di
m(x)
\]
where \(m:=\nabla f\) is the measure. Thus, we have
\[
|\nabla f_{\varepsilon }(y)|\le\int_{\rn }\rho _{\varepsilon }(y-x)
\di\|m\|(x).
\]
In particular, when \(f=\charfn{E}\) and \(f_{\varepsilon_k}=u_{k}\)
with $\varepsilon_k=1/k$, then \(m=\nabla \chi_{E}\) and
\[
|\nabla u_{k}(y)|\le\int_{\rn }\rho_{\varepsilon_k}(y-x)\di \|m\|(x)
\qquad\mbox{for all}\,\,\, y \in \rn.
\]
That is,
\begin{align*}
\int_{\rn }\abs{\nabla u_{k}(y)}\di y
&\le \int_{\rn} \int_{\rn
}\rho _{\varepsilon
}(y-x)\di \|m\|(x)\, dy \\
&= \int_{\rn} \int_{\rn }\rho _{\varepsilon }(y-x)\di y \di
\|m\|(x)\le\|m\|(\rn ).\qedhere
\end{align*}
\end{proof}

We recall that the $BV$ space, the space of functions of bounded
variation, in fact represents equivalence classes of functions so
that, when a function in a class is changed on a set of
$\mathcal{L}^{N}$-measure zero, it remains in this class. The same
is true for sets of finite perimeter because, by definition, the
characteristic function $\charfn{E}$ of a set of finite perimeter,
$E$, is a function of bounded variation.
Thus, it follows that $E$ may be altered by a set of
$\mathcal{L}^{N}$-measure zero and still determine the same
essential boundary $\po^{m}E$. Throughout, we will choose a
preferred representative for \(E \) and thereby adopt the  following
convention.

\begin{definition}\label{convention}
$E:=\{ y:D(E,y)=1\} \cup \partial^{m}E  $.
\end{definition}

\begin{definition}\label{dmfield}
A vector field $\FF\in L^p(\Omega;\rn)$, $1\le p\le \infty$, is
called a {\it divergence-measure field}, written as $\FF\in
\DM^p(\Omega)$, if $\mu:=\div \FF$ is a (signed) Radon measure with
finite total variation on $\Omega$ in the sense of distributions.
Thus, for \(\varphi \in C^{\infty}_{c}(\Omega) \), we have
\[
\mu (\varphi ):= \div\FF(\varphi )=-\int_{\Omega}\FF \cdot \nabla
\varphi\, dy.
\]
The total variation of \(\mu \) is a positive measure which, for any
open set \(W\), is  defined as
\begin{align*}
\norm{\mu }(W):
&= \sup\{\mu(\varphi)\, : \, \norm{ \varphi
}_{0; \Omega }\le1  ,\,\varphi \in C^{\infty}_{c}(W)   \}\\
&= \sup\Big\{\int_{\Omega}\FF \cdot \nabla
 \varphi \,dy\,  : \, \norm{ \varphi
}_{0; \Omega }\le1, \, \varphi \in C^{\infty}_{c}(W) \Big\}.
\end{align*}
A vector field $\FF\in L^p_{loc}(\Omega)$ means that, for any
$K\Subset\Omega$, $\FF\in \DM^p(K)$.
\end{definition}

\begin{definition}\label{trace}
Let \( \FF\in \DM^p(\Omega) \), \( 1\le p\le\infty\).
For an arbitrary measurable  set \(E\Subset\Omega \), the trace of the
normal
component of \( \FF \) on \( \partial E\) is a functional defined by
\begin{equation}\label{tracedisplay}
(T\FF)_{_{\partial E}}(\varphi )=\int_{E}\nabla \varphi \cdot \FF\,
dy +\int_{E}\varphi \di\mu
\end{equation}
for all test functions \( \varphi \in C_{c}^{\infty }(\Omega ) \).
Clearly, \( (T\FF)_{_{\partial E}}\) is a distribution defined on
\(\Omega\). Note that this definition assumes only that the set
\(E\) is measurable. Later,  we will provide an alternative
definition  when \(E\) is a set of finite perimeter (see Theorem
{\rm \ref{above}}).
\end{definition}

\begin{proposition}\label{sptiynbdry} Let $E\Subset\Omega$ be an
open set. Then \( \spt((T\FF)_{_{\partial E}})\subset\partial E\).
That is, if \( \psi \) and \( \varphi  \) are test functions in \(
\mathcal D(\Omega ) \) with \(\psi =\varphi \) on \( \partial E\),
then \( (T\FF)_{_{\partial E}}(\psi)= (T\FF)_{_{\partial
E}}(\varphi) \).
\end{proposition}

\begin{proof}
If the support were not contained in \( \partial E\), there would be
a point \( x_{0}\notin\partial E\) with \( x_{0}\in\spt(
(T\FF)_{_{\partial E}})\cap E\). This implies that, for each open
set \( U \) containing \( x_{0} \), there exists a test function \(
\varphi \in C_{c}^{\infty }(U\cap E)\) such that \(
(T\FF)_{_{\partial
 E}}(\varphi )\ne 0 \). Choose \( U \) so that \(
U\subset\rn\setminus
\partial E\). Let \( \FF_{\varepsilon} \) denote the mollification of \( \FF \)
(see Lemma \ref{mollifier-properties}). Then, since
$\spt(\FF_\varepsilon\varphi)\Subset E$,
\begin{align*}
0&
=\int_{E}\div(\FF_\varepsilon
\varphi) \,dy =(T\FF_{\varepsilon})_{_{\partial E}}(\varphi )\\
&=\int_{E}\FF_\varepsilon \cdot\nabla \varphi
+\int_{E}\varphi \, \div \FF_\varepsilon\, dy\\
&\to\int_{E}\FF\cdot\nabla \varphi\,dy +\int_{E}\varphi\di\mu
 =(T\FF)_{_{\partial E}}(\varphi )\ne0,
 \end{align*}
where we used \(\partial E\cap \spt(\varphi)=\emptyset\) in the
limit. Thus, we arrive at our desired contradiction.
\end{proof}

\begin{definition}\label{def:capacity}
For \(1\le p\le N \), the {\bf \(p \)-capacity} of an arbitrary set
\(A\Subset\rn \) is defined as
\begin{equation}\label{CAP}
\Cap_{p}(A
):=\inf\Big\{\int_{\Omega}\abs{\nabla\varphi}^{p}\,dy\Big\},
\end{equation}
where the infimum is taken over all test functions \(\varphi\in
C_{c}^{\infty }(\Omega)\) that are identically one in a neighborhood
of \(A \).  It is well known (cf. \cite{FZ}) that \(\Cap_p(A )=0 \)
for \(1<p<N\) implies that \(\H^{N-p+\varepsilon }(A) =0 \) for each
\(\varepsilon>0\) and that, conversely, if \(\H^{N-p}(A )<\infty \),
then \(\Cap_{p}(A )=0 \). In view of Remark {\rm \ref{precise}} and
Lemma {\rm \ref{capacity}}, it is easy to verify that the class of
competing functions in \eqref{CAP} can be enlarged to the Sobolev
space \(W^{1,p}(\Omega)\).
\end{definition}

\begin{remark}\label{fleming}
  The case of \(p=1\) requires special consideration. In {\rm 1957}, Fleming
  conjectured that \(\Cap_{1}(A )=0 \) if and only if \(\H^{N-1}(A )=0
  \). This was settled in the affirmative by Gustin \cite{Gustin} who proved the
  boxing inequality,  from which Fleming's conjecture
  easily follows (cf. \cite{Fleming}).
\end{remark}
The next result is basic (cf. \cite{Ziemer2,zel,zexl}).

\begin{proposition}\label{capProp}
Let \(\gamma_{p}\) be the \(p\)-capacity defined as in the previous
definition. Then
\begin{enumerate}[\rm (i)]
\item If \(E_{k} \subset \rn \) is a sequence of arbitrary sets, then
\[
\gamma_{p} (\liminf _{k \to \infty }E_{k})\le \liminf_{k\to \infty }
\gamma_{p} (E_{k});
\]
\item If \(E_{1}\subset E_{2}\subset \cdots\) are arbitrary sets, then
\[
\gamma_{p} \left(\medcup_{k=1}^{\infty }E_{k}\right)=\lim_{k\to
\infty } \gamma_{p} (E_{k});
\]
\item If \(K_{1}\supset K_{2}\supset \cdots\) are compact sets, then
\[
\gamma_{p} \left(\medcap_{k=1}^{\infty }K_{k}\right)=\lim_{k\to
\infty } \gamma_{p} (K_{k});
\]
\item
If \(\{E_{k} \}\) is a sequence of Borel sets, then
\[
\gamma_{p}\left(\medcup_{k=1}^{\infty }E_{k}\right)\le\lim_{k\to
\infty } \gamma_{p} (E_{k});
\]

\item If \(A\subset \rn \) is a Suslin set, then
\[
\sup\{\gamma_{p}(K): K^{\text{compact}}\subset A\}=\inf\{\gamma_{p}
(U): U^{\text{open}}\supset A \}.
\]
\end{enumerate}
Any set function, \(\gamma \), satisfying conditions {\rm (i)--(iv)}
is called a {\bf true capacity} in the sense of Choquet and a set
\(A\) satisfying condition {\rm (v)} is said to be {\bf
  $\gamma$-capacitable}.
\end{proposition}

\begin{remark}\label{capsobolev}
One of the main reasons for studying the
  capacity is its important role in the development of
  Sobolev theory. It was first shown in \cite{FZ} that every function
  \(u \in W^{1,p}(\Omega)\) has a Lebesgue point \(\gamma_{p}\)-a.e. In
  particular, in view of Remark {\rm \ref{fleming}}, this implies that a
  function \(u\in W^{1,1}(\Omega)\) has a Lebesgue point everywhere
  except for an exceptional set \(E\) with \(\H^{N-1} (E)=0\). In case
  \(u\in BV(\Omega)\), we have a slightly weaker statement than the
  corresponding one for \(u\in W^{1,p}(\Omega) \):
\[
\lim_{r\to0}\intave{B(x,r)}u(y) \di  y=u(x)\qquad\text{for
\(\H^{N-1} \)-a.e. \(x\in \Omega\).}
\]

It turns out that the Sobolev space is the {\it perfect functional
completion} of the space \(C^{\infty}_{c}(\Omega) \) relative to the
\(p\)-capacity.  See \cite{AS1} where the concept of perfect
functional completion was initiated and developed.
\end{remark}

\begin{lemma}\label{capacity}
If \( \FF\in \DM^{p}_{loc}(\rn)\), \(1\le p\le\infty\), then \(
\norm{\div \FF} (B)=0 \) whenever \(B\Subset\Omega\) is a Borel set
with \( \Cap_{q}(B)=0 \), \(q:=p/(p-1) \). In particular, when \(
p=\infty \) (i.e. \(q=1\)), then \( \norm{\div \FF}(B)=0 \).
\end{lemma}

\begin{proof}
Because of the inner regularity of \(\gamma_{q}\) and condition (v)
of Proposition \ref{capProp}, it suffices to show that $\mu(K)=0$
for any compact set $K \subset B$, where $\mu:=\div\FF$. Since
\(\gamma_{q}(K)=0\), then there exists a sequence of test functions
\(\varphi_{k}\in C^{\infty}_{c}(\Omega)\) such that
\begin{enumerate}[(i)]
\item $\varphi_{k}=1 \quad\text{on \(K\)}$;\\
\item $\norm{\nabla \varphi_{k}}_{q}\to0$;\\
\item $\varphi _{k} (y)\to0$ for all $y\in \Omega$  except those in some
set \(A \subset\Omega\) with \(\gamma_{q}(A)=0\).
\end{enumerate}
To see that such a sequence \(\varphi _{k}\) exists, we proceed as
follows:\\

{\bf Case 1: $q>1$.} Since \(\gamma_{q}(K)=0\), we may choose
\(u_{i}\in C_{c}^{\infty}(\Omega)\) with \(u_{i}=1\) near \(K\) and
\(\norm{\nabla u_{i}}_{q}\to0\). Then Sobolev's inequality implies
\[
\norm{u_{i}}_{r}\le\; C \norm{\nabla
  u_{i}}_{q}\qquad\text{for \(r=Nq/N-q\)}.
\]
Thus, it follows that, for a subsequence, \(u_{i}(x)\to0\) for a.e.
\(x\). In fact, an application of Mazur's theorem shows that the
space of convex combinations of  \(W^{1,p}(\Omega\) is strongly
closed and therefore, for a suitable subsequence, the sequence of
finite convex combinations of the \(u_{i}\),
 \,\text{say}
\(v_{k}:= \sum a_{i} u_{i} \) with \(\sum_{i=1}^{k_{i}}a_{i}=1\), it
follows that \(v_{k}(x)\to0\) for all \(x\) except those in an
exceptional set \(A\) where \(\gamma_{q} (A)=0\) (see \cite{FZ}, p.
156). Taking \(\varphi_{i}:= v_{i}\) for this subsequence
establishes a sequence
 satisfying conditions (i)--(iii) above.

{\bf Case 2: \(q=1\) (\(p=\infty \))}. Then the argument is modified
by considering the following variational problem:
\begin{align}\label{varprob}
\Gamma_{k,q}(K):= \inf\Big\{\int_{U_{k} }|\nabla u_{k} |^{q}\,dy\,
:\, u_{k} \in W^{1,q}(\rn ),\, u_{k} =1 \;\text{on}\;
\overline{U_{k}},\, \spt (u_{k}) \subset W_{k}\Big\}.
\end{align}

Let \(W_{k}\supset \overline{U_{k}} \supset U_{k }\cdots\supset K \)
with $U_k$ open and
\begin{equation}\label{capto0}
\gamma_{p}(\overline{W_{k}})=\gamma_{p}(W_{k})<\varepsilon_{k}\to0.
\end{equation}
We will also assume \(W_{k} \) to be nested: \(W_{1}\supset
W_{2}\cdots\)\,. For each \(k\), let \(\overline{W_{k}}\) be a
smoothly bounded set to find that \(\medcap \overline{W_{k}}=K.\)
Consequently, we may take \(\varphi_{k}\) to be a suitably small
mollification of \(\charfn{W_{k}}\) to see that conditions \({\rm
(i)}\)--\({\rm (iii)}\) above are satisfied.

To prove our lemma, it suffices to show that, if $K$ is any compact
set satisfying $\Cap_{q}(K)=0$, then $\mu(K)=0$. Since
\begin{enumerate} [(a)]
\item  \(\varphi _{k} =1\) on \(K\);
\item\(\varphi_{k}\to0\)\quad \(\gamma_{q}\)-a.e.;
\item \(\norm{
\nabla \varphi_{k}}_{q;\rn }\to 0\);
\item
\(\norm{\mu}(W)=\sup\{ \int_{W}|\FF\cdot  \nabla \varphi|\,
dy:\varphi \in C^{\infty}_{c}(W),\abs{\varphi} \le1\}\) \,\, for any
open set \(W \Subset\Omega\),
\end{enumerate}
it follows that
 \begin{align*}
\lim_{k\to \infty} \big(\int_{K} \varphi _{k} \di
\mu+\int_{U_{k}\setminus
       K}\varphi _{k}\,\di\mu\big)
   &=\mu (K)+\lim_{k \to \infty }\int_{U_{k}\setminus
       K}\varphi _{k}\,\di\mu  \\
   &\le\mu (K)+ \lim_{k \to \infty }\int_{U_{k}\setminus
       K}\varphi _{k}\,\di \norm{\mu }  \\
   &\leq\lim_{k\to \infty}\left|\int_{\Omega}\varphi_{k}\, \div
\FF\right
   |\le\lim_{k\to \infty}\int_{\Omega}|\FF\cdot\nabla \varphi_{k}| dx\\
 &\le \norm{\FF}_{p} \lim_{k\to \infty} \norm{\nabla \varphi_{k}}_{q}
   \le  \,\norm{\FF}_{p}\cdot 0,
  \end{align*}
thus  obtaining $\mu(K)=0$, as desired.
\end{proof}

\begin{corollary}\label{muis0}
If \(\FF\in \mathcal{DM}^{p}_{loc}(\Omega)\) for $1<p\le N$ with
\(\div \FF=:\mu \) and if \( \H^{N-q}(A)<\infty\) for
$A\Subset\Omega$ with $q=p/(p-1)$, then \(\Cap_{q}(A)=0 \) and hence
\(\norm{\mu}(A)=0\).
\end{corollary}

\begin{remark}
If \(\FF\in \mathcal{DM}^{p}(\rn)\), \(1\le p\le \infty \), with
\(\div \FF=\mu\). Then, in view of the fact that \(\varphi \) is
defined \(\H^{N-q}\)-a.e. and therefore \(\mu \)-a.e., with
\(\varphi \in  W^{1,q}(\rn )\), it follows that the integral
\[
\int_{\rn }\varphi \di  \mu\,
\]
is defined and is meaningful.
\end{remark}

\begin{example}[Chen-Frid \cite{CF1}] Denote \(U\) the open unit square in
\(\rii\) that has one of its sides contained in the line segment
\[
L:= \{y=(y_{1},y_{2}):y_{1}=y_{2} \}\cap \partial U.
\]
Define a field \(\FF:\rii \setminus L\to \rii \) by
$$
\FF(y)=
\FF(y_1,y_2)=(\sin\big(\frac{1}{y_{1}-y_{2}}\big),-\sin\big(\frac{1}
  {y_{1}-y_{2}}\big)).
$$
\end{example}
Clearly, \(\FF\in L^{\infty }(\rii)\), and a simple calculation
reveals that \(\div \FF=0\) in \(\rn \setminus L\). Then \(\FF\)
belongs to \(\mathcal{DM}^{\infty }(\rii \)); but the field is
singular on one side, \(L\), of \(\partial U\) and therefore,
\(\FF\) is undefined on \(\partial U\); it has no trace on
\(\partial U\) in the classical sense. Note also that the points of
\(L\) are all essential singularities of \(\FF\) because the
following limit does not exist:
\[
\lim_{y\to x}\FF(y)\qquad\text{for $y\in \R^2 \setminus L$, \(x\in
L\)},
\]
and therefore the normal trace of \(\FF\) on \(\partial U\) is given
by
\[
\lim_{t\to 0} \int_{\partial U_{t}} \FF(y)\cdot \nnu(y)\di \H^{1}
(y)= \lim_{t\to0}\int_{U_{t}} \div \FF \di y =\lim_{t\to0} 0=0,
\]
where \(U_{t}:= \{y\in U :\dist(y,\partial U)>t) \}\). Thus, we have
shown the following:

\begin{enumerate}[(i)]
\item \(\FF\) is an element of \(\mathcal{DM}^{\infty }(\R^2) \), while
each component function of $\FF$ is not in \(BV(\R^2)\);
\item \(\FF\) has an essential singularity at each point of \(L\) and
      therefore cannot be defined on \(L\);
\item As we will see later, Theorem \ref{above}, \(\FF\) has a
      weak normal trace on \(L\) which is sufficient for the
      Gauss-Green theorem to hold.
\end{enumerate}

For more properties of the spaces $\DM^p$ of divergence-measure
vector fields, see Chen-Frid \cite{CF1,CF-CMP}.

The following theorem provides a product rule for the case
$p=\infty$. For the sake of completeness, we will include its proof,
which is slightly different from that given in \cite{CF1}. We denote
by \(\{g_{k}\}\) a sequence of \(C^{\infty}_{c}(\rn )\)
mollifications with the property that \(g_{k} \to g \) in
\(L^{1}(\rn )\) and such that \(\norm{\nabla g_{k}}\to \norm{\nabla
g}\) (cf. \cite{Ziemer2}, p.500).

\begin{theorem}[Chen-Frid \cite{CF1}]\label{productrule}
Let $\FF \in \mathcal{DM}^\infty(\rn)$ and \(g\in BV(\rn )\).
Then
\begin{equation}\label{prod}
{\div} (g \FF) = g^{*}\,{\div}\FF + \overline{F \cdot \nabla
g},
\end{equation}
where $\overline{F \cdot \nabla g}$ denotes the weak\(^{\,*}\) limit
of the measures $\FF \cdot\nabla g_k$ and \(g^{*}\) denotes the
limit of the mollifiers of \(g\) (cf. \cite{Ziemer2}).
\end{theorem}

\begin{proof}
Let \(\FF_{\varepsilon}\) be the mollification of \(\FF\) and set
$\mu:=\div \FF$. Since \(\FF_{\varepsilon}\) are smooth, the
classical product rule yields
\begin{equation}\label{eq:1}
\div (g_k \FF_{\varepsilon})
=g_{k}\,\div \FF_{\varepsilon }+\FF_{\varepsilon }\cdot \nabla
g_{k}.
\end{equation}
First, we note that \(\div \FF_{\varepsilon }=(\div
\FF)_{\varepsilon }=\mu _{\varepsilon }\;
\stackrel{*}{\rightharpoonup} \, \mu \) in $\mathcal{M}(\rn)$ as
$\varepsilon\to 0$. Since \(g_k \FF_{\varepsilon } \;{\to}\;g_k \FF
\) in \(L^{1}_{loc} (\rn) \) as $\varepsilon\to 0$, we obtain from
\eqref{eq:1} that, in the sense of distributions,
\begin{equation}\label{muchodolor}
\FF \cdot \nabla g_k= \div (g_k \FF) - \mu(g_k).
\end{equation}
Owing to the fact that \(\FF\in L^{\infty} \), we see that
\(\FF\cdot \nabla g_{k}\) is a bounded sequence in \(L^{1} (\rn) \)
and hence there is a subsequence such that \( \FF\cdot \nabla
g_{k}\) converges weak\(^{\;*} \) to some measure, denoted by
$\overline{\FF \cdot \nabla g}$. Letting $k \to \infty$ in
(\ref{muchodolor}) yields
\begin{equation}
\overline{\FF \cdot \nabla g}= \div (g \FF) - \mu(g^{*}).
\end{equation}
\qedhere
\end{proof}

The next result, Federer's coarea formula, will be of critical
importance to us in the sequel.

\begin{theorem}[\cite{Fe}, Theorem 3.1]\label{coarea}
Suppose \(X\) and \(Y\) are Riemannian manifolds of dimension \(N\)
and \(k\) respectively, with \(N\ge k\). If $f\colon X \to Y$ is a
Lipschitz map, then
\begin{equation} \label{caformula}
\int_{X}g(y)Jf(y)\;d\H^{N}(y)=\int_{\R^{k}}\Big\{\int_{f^{-1}(x)}g(y)d\H^{N-k}
(y)\Big\}\, d\H^{k}(x)
\end{equation}
whenever $g\colon X\to\R$ is $\H^{N}$-integrable. Here, $Jf(y)$
denotes the \(k\)-dimensional Jacobian of \(f\) at \(y\), namely,
the norm of the differential of \(f\) at \(y\), \(df(y)\).
Alternatively, it is the square root of the sum of the squares of
the determinants  of the $k\times k$ minors of the differential of
$f$ at $y$.
\end{theorem}

\begin{lemma}\label{transversal}
Let $u\colon \rn \to \R$ be a Lipschitz function and let \(A \subset
\rn\) be a set of measure zero. Then
\[
\H^{N-1} (u^{-1}(s) \cap A )=0\qquad\text{for almost all \(s\)}.
\]
\end{lemma}
This can be directly seen from the coarea formula:
\[
0=\int_{A}|\nabla u(y)|\di y =\int_{\ri} \H^{N-1}(A\cap u^{-1}(s))
\di s.
\]

One of the fundamental results of geometric measure theory is that
any set of finite perimeter possesses a measure-theoretic interior
normal which is suitably general to ensure the validity of the
Gauss-Green theorem.

\begin{theorem}[DeGiorgi-Federer
\cite{Georgi1,Georgi2,Federer1,Federer2}]\label{ggbasic}
If \(E\) has finite perimeter, then, with \(\nnu\) denoting the
interior unit normal,
\[
\int_{E}\div \FF\di y=-\int_{\partial^{*}E}\FF(y)\cdot\nnu (y)\di
\H^{N-1}(y)
\]
whenever \( \FF\colon\rn\to\rn \) is Lipschitz.
\end{theorem}

The DeGiorgi-Federer result shows that integration by parts holds on a
very large and
rich family of sets, but only for fields, \(\FF \), that are
Lipschitz. As we explained in the introduction,  the Gauss-Green
formula for $BV$ vector fields over sets of finite perimeter was
treated by Maz'ja \cite{BM2}
and  Volpert \cite{Volpert}. We contrast their work with that of the
following result by Fuglede.

\begin{theorem}[Fuglede \cite{fuglede}]\label{fugledeaaa}
Let $\FF\in \DM^{p}(\R^{N})$, $1\le p\le\infty$.  Then there exists
a function $g: \R^{N}\to \R$ with $g\in L^{q}$,
$\frac{1}{q}+\frac{1}{p}=1$, such that
\begin{equation}
  \label{fugledegg}
\int_{E}{\div}\FF = -\int_{\partial ^{*}E}\FF(y)\cdot\nnu
(y)\,d\H^{N-1}(y)
\end{equation}
for all sets of finite perimeter, $E$, except possibly those for
which
\[
\int_{\partial ^{*}E}g(y)\,d\H^{N-1}(y)=\infty .
\]
\end{theorem}

The following, which is a direct consequence of Fuglede's result,
will be of great use to us. Suppose that \(u\colon \rn \to \ri\) is
Lipschitz. For \(s<t\), consider the ``annulus''
\(\mathcal{A}_{s;t}:= \{x:s<u(s)\le t \}\) determined by \(u\).
Then, by appealing to the coarea formula, we see that
\(\mathcal{A}_{s;t} \) is a set finite perimeter for almost all
\(s<t\). Moreover, again appealing  to the coarea formula, we see
that, for almost all \(s<t\),
\begin{equation}
  \label{goodann}
(\div \FF) (A_{s;t}):=\int_{\mathcal{A}_{s;t} }{\div}  \FF
=-\int_{\partial^{*}\mathcal{A}_{s;t} }\FF(y)\cdot\nnu
(y)\,d\H^{N-1}(y).
\end{equation}

One of the main objectives of this paper is to demonstrate that,
when \(\FF\in \DM^{\infty}(\rn)\), we can extend Fuglede's result by
showing that \eqref{fugledegg} and \eqref{goodann} hold for all sets
of finite perimeter,  not merely  for ``almost all'' sets in the
sense of Fuglede \cite{fuglede}. Although we don't employ Fuglede's
theorem directly, his result provided the motivation and insight for
the development of our method.

The Gauss-Green formula for bounded divergence-measure fields over
sets of finite perimeter was first obtained in Chen-Torres
\cite{ChenTorres}. The product rule from Lemma \ref{productrule}
was used to prove that
\begin{equation}
  \label{chent}
(\div \FF)(E^{1}) :=  \int_{E^{1}}\div \FF=-\int_{\partial^*E
}2\overline{\chi_{E}\FF\cdot \nabla u_{E}},
\end{equation}
where $2\overline{\chi_{E}\FF\cdot \nabla u_{E}}$ is the weak* limit
 of the measures $2\chi_E \FF \cdot \nabla u_k$. One of the main
objectives of this paper is to obtain the trace measure as the limit
of normal traces over smooth boundaries that approximate $\po^* E$.

\section{The Normal Trace and the Gauss-Green
 Formula for $\DM^\infty$ fields
over smoothly bounded sets}

In this section we develop a method of obtaining the normal trace
and thereby obtain the Gauss-Green formula for a bounded
divergence-measure field over any smoothly bounded set. This
provides the foundation for the development for the general case.

\begin{definition}
  Given a compact \(C^{1}\)-manifold, \(M\), we define the {\bf
    exterior} determined by \(M\) to be that (connected) component,
  \(\mathcal{U}\), of \(\rn \setminus M\) which is unbounded. The {\bf
interior} determined by \(M\), \(U\), is defined to be everything
else in
  the complement of \(M\); namely,
\[
U=\medcup_{k=1}^{\infty }B_{k},\quad\text{\(B_{k} \subset
  \rn \setminus M\) a bounded component.}
\]
Thus,
\[
\rn \setminus M=\mathcal{U} \cup\big(\medcup_{k=1}^{\infty
}B_{k}\big)=\mathcal{U}\cup U.
\]
\end{definition}

\begin{theorem}\label{GGforC1}
Let \(U\subset\rn\) be the interior determined by a compact, \(
C^{1}\), manifold \(M\) of dimension \(N-1\) with
\(\H^{N-1}(M)<\infty\). Then, for any \(
\FF\in\DM^\infty_{loc}(\rn)\), there exist  a signed measure \(
\sigma \) supported on \(
\partial U=M \) such that \( \sigma <
<\mathcal{H}^{N-1}\rtangle\partial U\) and a function
\(\mathscr{F}_{i}\cdot\nnu : \partial U \to \R\) such that
\begin{equation}\label{weakgg}
\mu (U):=   (\div \FF)(U)=-\sigma (\partial U)=-\int_{\po U
}(\mathscr{F}_{i}\cdot \nnu)(y)\di \H^{N-1}(y)
\end{equation}
and
$$
\norm{\mathscr{F}_{i}\cdot
\nnu}_{\infty}\leq C \norm{\FF}_{\infty},
$$
where $C$ is a constant depending only on \(N \) and $U$.
\end{theorem}

With \(\nnu(y)\) denoting the interior unit normal to \(M\) at \(y
\), we may regard \(\mathscr{F}_{i}\cdot\nnu\) as the {\bf interior
normal trace} of \(\FF\) on \(\partial U\) and thus write
\[
(\mathscr{F} _{i}\cdot \nnu)(y)=\FF(y)\cdot\nnu (y).
\]
Hence, with this convention, it is convenient to abuse the notation
and thus write \eqref{weakgg} as

\begin{equation}\label{abusetrace}
\mu (U)=\int_{U}\div \FF=-\int_{\partial U}\FF(y)\cdot\nnu (y)\,\di
\mathcal{H}^{N-1}(y),
\end{equation}
while bearing in mind that, since \(\FF\) is merely a measurable
field and thus defined only up to a Lebesgue null set, it may not
even be defined on \( \partial U\). We use the term ``interior
normal trace'' to suggest that \(\mathscr{F} _{i} \cdot \nnu\) is
determined by the behavior of \(\FF\) in the interior determined by
the manifold \(M\). The proof will reveal that, in a similar way, it
is possible to define the concept of ``exterior normal trace''. This
will be discussed more fully
below in Theorem \ref{main}.\\

The next example shows that our general trace theorem remains valid
even though the field \(\FF\) is singular on an open subset of
\(\partial U\).

\begin{lemma}\label{Fuglede}
Let \(\FF\in \DM^\infty_{loc}(\rn)\) whose distributional divergence
is a measure \( \mu \), and let \( \FF_{\varepsilon } \) be a
mollification of \(\FF\). Then, because \(\FF_{\varepsilon}\) is
smooth (in particular, Lipschitz), the classical divergence theorem
holds whenever \(E\Subset\rn \) is a set of finite perimeter,
namely,
\begin{equation}\label{ggforsmoothF}
\int_{E}\div \FF_{\varepsilon }=-\int_{\partial^ *E}\FF_{\varepsilon
}(y)\cdot\nnu(y) \di \H^{N-1}(y).
\end{equation}
If, in addition,  we assume the following two conditions:
\begin{enumerate}[\rm (i) ]
\item $\FF_{\varepsilon }\to \FF $\,\, \( \H^{N-1} \)-a.e. on \(
\partial^ *E \),
\item \( \mu (\partial E)=0 \),
\end{enumerate}
then
\begin{equation}\label{gg}
\mu (E)=-\int_{\partial^ *E}\FF(y)\cdot\nnu(y) \di \H^{N-1}(y).
\end{equation}
\end{lemma}

The importance of this result is that, with assumptions (i) and
(ii), we obtain the Gauss-Green theorem for all sets of finite
perimeter whenever \(\FF \) is a bounded, measurable vector field
with \(\div \FF=\mu \). As stated earlier, our main objective is to
obtain the same result without assuming (i) and (ii), by defining a
suitable notion of normal trace for \(\FF\) on \(\partial^{*}E \).

\begin{proof}
Since \(\mu _{\varepsilon }:= \div \FF_{\varepsilon
}\stackrel{*}{\rightharpoonup} \, \div \FF=\mu \)
in $\mathcal{M}(\rn)$ and using the fact that \( \FF_{\varepsilon}
\rtangle \partial^{*}E\) is uniformly bounded, we obtain from
\eqref{ggforsmoothF} that
\begin{eqnarray*}
&&\mu _{\varepsilon}(E) :=\int_{E}\div \FF_{\varepsilon}
=-\int_{\partial^*E}\FF_\varepsilon (y)
 \cdot\nnu(y) \di \H^{N-1}(y)
 \to -\int_{\partial^
*E}\FF(y)\cdot\nnu(y) \di \H^{N-1}(y), \\
&&\mu _{\varepsilon }(E)\to\mu (E)\quad \text{(by assumption {\rm
(ii)})}.
\end{eqnarray*}
This establishes our result.
\end{proof}

\begin{proof}[Proof of Theorem {\rm \ref{GGforC1}}]
According to \cite{Whitney}, there exist a \(C^{1}\) unit vector
  field \(V\) defined on \( \partial U\) and a number
  \(\delta>0 \) such that \(V\) is close to the interior normal on
  \(\partial U\) and that, with the line segment joining \(p
  \in \partial U\) and \(q \in U\) defined as
\[
\lambda(p):=\{q:q=tV(p)+(1-t)p,\, 0\le t\le1 \},
\]
then
\[
\Lambda^{*}(p):=\lambda (p)\cap B(p,\delta)\in U.
\]
We think of \(\Lambda^{*}(p)\) as a {\bf quasi-normal} to \(\partial
U\) at \(p\). Moreover, as \(p\) ranges over \(\partial U\), the
quasi-normals, \( \Lambda^{*}(p) \), fill out a neighborhood \(
U^{*} \) of \( \partial U\) in a one-to-one way. That is,
\[
U^*=\medcup_{p\in \partial U}\Lambda^{*}(p)\qquad\text{ with
\(\Lambda^{*} (p)\cap \Lambda^{*}(q)=\emptyset\) when \(p \ne q\)}.
\footnote{Here, we could have defined \(E^{*}_{e}:= (\rn
  \setminus \partial U)\cap E^{*}\) to develop the notion of the
  exterior normal} \]
The mapping \( \pi\colon U^{*}\to \partial U\), which can be
considered as the projection of \(U^{*}\) onto \(\partial U\) along
the quasi-normal, \(\Lambda ^{*}\), defined by
\[
\pi(q)=p\qquad\text{if \( q\in \Lambda^{*}_{p},\)}
\]
is clearly of class \( C^{1} \) and thus, so is the mapping \( \psi
\colon U^{*}\to \R\) defined by
\[
\psi(q):=\abs{q-p} \qquad \text{where \(\pi(q)=p\)}.
\]
Consequently, Sard's theorem implies that the critical values, \(
\mathbb{V} \), of \( \psi \) are of measure zero. Thus, the implicit
function theorem implies that \(\partial U_{t}:= \psi^{-1}(t)\) is a
\( C^{1} \) manifold for almost all
\(t\in(0,\delta)\setminus\mathbb{V}\). The manifold \(
\partial U_{t} \) can be considered as a deformation of
\( \partial U\) along the line segments \( \lambda (p) \). Thus, the
sets \( U_{t}:=\{\psi >t\} \) are open subsets of \( U\) with smooth
boundaries and \(\partial U_{t}=\psi ^{-1}(t) \) for a.e. \(t\).
Observe that \( \psi ^{-1}(t)=\{q\in U :q=t(V(p)-p),\, p\in
\partial U\} \). Since \(V\in C^{1}\) and \(\partial U\) is compact,
it follows that \( V \) is Lipschitz with Lipschitz constant, say \(
C_c \), and therefore that
\begin{equation}\label{uniformbound}
\mathcal{H}^{N-1}(\psi ^{-1}(t))\le C_c\, \mathcal{H}^{N-1}(\partial
U).
\end{equation}

To see this, consider a local coordinate system on \( \partial U\)
expressed as the inverse of its projection onto the tangent plane at
a point \( p\in \partial U\). Consider \( T_{p}(\partial U) \) as a
subspace in \( \rn \) with the notation \( p=(p',0)\in\rn \) as
identified with \( p'\in \R^{N-1} \). Thus, \( g\colon
T_{p}(\partial U)\to \partial U\) will render \(
\partial U\) as the graph of \(g \) so that
\(g(p')=(p',y(p')) \), where \(y\in C^{1}\) and is the ``height''
function. Now define \( h\colon \partial U\to
\partial U_{t} \) by
$$
 h(p',y(p'))=tV(p',y(p'))+(1-t)(p',y(p')).
$$
Clearly, all the derivatives of \( h \circ g \) are uniformly
bounded on compact subsets of \( T_{p}(\partial U) \) and then
\eqref{uniformbound} is evident.  An application of Lemma
\ref{Fuglede} shows that
\[
\mu (U_{t})=-\int_{\partial U_{t}}\FF(y)\cdot\nnu (y)\,
d\mathcal{H}^{N-1}(y)\quad\text{for a.e. \( t>0 \)},
\]
where $\mu:=\div \FF$ is a Radon measure. Thus, the function \(
\FF\cdot\nnu\rtangle
\partial U_{t} \) induces a measure on \( \rn\). That is, for any
Borel set \( B\subset\rn \), define
\[ \sigma_{t}(B):=\int_{B\cap\partial U_{t}}\FF(y)\cdot\nnu (y)\,
d\mathcal{H}^{N-1}(y).
\]
Since \( \FF\in L^{\infty } \) and \( \mathcal{H}^{N-1}(\partial
U)<\infty  \), we see from \eqref{uniformbound} that \(
\mathcal{H}^{N-1}(\partial U_{t})\le C \) for some \( C>0 \), which
yields that the measures \( \sigma _{t} \), \( t>0 \), form a
bounded set in $\mathcal{M}(\rn)$. Hence, there exist a sequence \(
\{t_{k}\}\to0 \) and  Radon measures \( \sigma\), \(\sigma^{+} \),
and \(\sigma^{-}\) with $\sigma=\sigma^+-\sigma^-$ such that
\begin{equation} \label{sigmaconverge}
(\sigma_{t_k}^+, \sigma_{t_k}^-, \sigma _{t_{k}}) \,
\stackrel{*}{\rightharpoonup} \,(\sigma^+, \sigma^-, \sigma)
\quad\qquad\mbox{in}\,\, \mathcal{M}(\rn).
\end{equation}

Now we show that \(\sigma^+\) is supported on \(\partial U\). On the
contrary, let \(x\in \textnormal{spt} (\sigma^+) \setminus\partial
U\) and choose \(B(x,r)\) such that \( B(x,r)\cap\partial
U=\emptyset\). Since \( x\in \textnormal{spt} (\sigma^+)\), there
exists \( \varphi \in C(B(x,r)) \) such that \( \int\varphi \,
d\sigma^+:=\sigma^+ (\varphi )\ne0 \). Then, since \(\varphi\) is
continuous, we find that \(\sigma^+_{t}(\varphi )\to\sigma^+
(\varphi )\ne0 \). This implies that \( \sigma _{t}^+(\varphi
)\ne0\) for all small \( t>0\), which leads to a contradiction since
\( \partial U_{t}\cap B(x,r)=\emptyset \) and \(\textnormal{spt}
(\sigma_{t}^+)\subset\partial U_{t}\).

Clearly,
\begin{equation}
\sigma_{t_{k}}^+(\partial U_{t_{k}} )\to\sigma^+(\partial U),
\end{equation}
since
\[
\liminf_{k\to\infty }\sigma_{t_{k}}^+(\partial U_{t_{k}}
)=\liminf_{k\to\infty} \sigma_{t_{k}}^+(\rn )\ge \sigma^+
(\rn)=\sigma^+ (\partial U),
\]
whereas
\[
\limsup_{k\to\infty }\sigma_{t_{k}}^+(\partial U_{t_{k}}
)=\limsup_{k\to\infty }\sigma_{t_{k}}^+(K )\le \sigma^+(K)=\sigma^+
(\partial U)
\]
for any closed set \( K \) that contains \(U\), especially when \(
K=\rn \). In the same way, we prove that \begin{equation}
\sigma_{t_{k}}^-(\partial U_{t_{k}} )\to \sigma^-(\partial U),
\end{equation}
and hence
\begin{equation}\label{goodlimit}
\sigma_{t_{k}}(\partial U_{t_{k}} )\to\sigma(\partial U).
\end{equation}

Also note that the measure \(\sigma\) is well-defined, which can be
seen as follows: Let \(U_{t'_{k}}\) be another sequence of open sets
with smooth boundaries defined by \(U_{t_{k}'}:=\{\psi
>t_{k}'\} \) for which Lemma \ref{Fuglede} applies. Then, assuming
that  \( t_{k}>t'_{k} \) for all \( k \), we have
\begin{align*}
\mu (U_{t'_{k}}\setminus U_{t_{k}})
&=\int_{U_{t'_{k}}} \div \FF -\int_{U_{t_{k}}} \div \FF \\
&=-\int_{\partial U_{t'_{k}}}\FF(y)\cdot\nnu (y)\, d\H^{N-1}(y)
+\int_{\partial
U_{t_{k}}}\FF(y) \cdot \nnu(y)\, d\H^{N-1}(y)\\
&=-\sigma_{t'_{k}}(\partial U_{t'_{k}})+\sigma_{t_{k}}(\partial
U_{t_{k}}).
\end{align*}
Since \(U_{t'_{k}}\setminus U_{t_{k}} \) is a monotone sequence of
sets with \(U_{t'_{k}}\setminus U_{t_{k}}\to\emptyset \), it follows
that \( \mu(U_{t'_{k}}\setminus U_{t_{k}})\to0 \) and therefore that
\(\sigma_{t'_{k}}(\partial U_{t'_{k}})-\sigma_{t_{k}}(\partial
U_{t_{k}} )\to0\), which shows that \( \sigma  \) is well-defined.

To show that \(\sigma < <\mathcal{H}^{N-1}\rtangle\partial U\), let
\( A\subset\partial U\) with \( \mathcal{H}^{N-1}(A)=0 \). Then
there exists an open set $G \Subset \rn$, \( G\supset A \), such
that \( G:=\medcup_{k=1}^{N}B_{k}(r_{k}) \) with
\(\displaystyle\sum_{k}r_{k}^{N-1}<\varepsilon\) and $\H^{N-1}(G
\cap \po U)< \ve$. Therefore, we have
\[
\norm{\sigma _{t}}(G)\leq\int_{G\cap M_{t}}\abs{\FF(y)\cdot\nnu(y)}
\,d\mathcal{H}^{N-1} (y) \le \norm{\FF}_{L^\infty(G)
}\mathcal{H}^{N-1}(G\cap M_{t})<\varepsilon C
\norm{\FF}_{L^\infty(G)},
\]
that is, for some \(C>0\), \( \norm{\sigma_{t}} (G)< \varepsilon
C\norm{\FF}_{L^\infty(G)}\), which leads to the conclusion \(
\norm{\sigma}(A) =0\) as desired.

We note that, since the sets $U_{t_k}$ are increasing, we have
\begin{equation}\label{good2}
\mu(U_{t_k}) \to \mu(U).
\end{equation}
Thus, using \eqref{goodlimit} and \eqref{good2} and  sending $k \to
\infty$ in $\mu(U_{t_k})= -\sigma_{t_k}(\partial U_{t_k})$ yield
\begin{equation}
\mu(U)= -\sigma(\partial U).\qedhere
\end{equation}
\end{proof}

\begin{remark} From the proof above, it can be seen that,
if \(U\) were of class \(C^{2}\), then the interior normals to
\(\partial U\) themselves would not intersect in a
  sufficiently small neighborhood of \(\partial U\).
\end{remark}

\section{Almost One-Sided Smooth Approximation of
Sets of Finite Perimeter}

We now proceed to establish a fundamental approximation theorem for
a set of finite perimeter by a family of sets with smooth boundary
essentially from the measure-theoretic interior of the set with
respect to
any Radon measure that is absolutely continuous with respect to
$\mathcal{H}^{N-1}$.
That is, we prove that, for any Radon measure
$\mu$ on $\rn$ such that $\mu<<\mathcal{H}^{N-1}$, the superlevel
sets of the mollifications of the characteristic functions of sets
of finite perimeter provide an approximation by smooth sets which
are \(\|\mu\|\)-almost contained in the measure-theoretic interior
of \(E\). This allows us to employ Theorem \ref{GGforC1} which,
after passage to a limit, leads to our main result, Theorem
\ref{main}.

\begin{lemma}\label{general} Let $\mu$ be a
Radon measure on \(\rn \) such that \(\mu < < \H^{N-1}\). Let $E$ be
a set of finite perimeter, and let \(u_{k}\) be the mollification of
\(\charfn{E} \). Then, for any \(t \in (0,1) \) and $A_{k;t} := \{y:
u_{k}(y)>t\}$, there exist $\ve = \ve(t)$ and
\(k^{*}=k^{*}(\varepsilon,t)\) such that
\begin{enumerate}[\rm(i)]
\item$ {\|\mu\| }(A_{k;t}\setminus E)<\varepsilon
\quad\text{ \rm for all \(0<t <\frac{1}{2}\) and all \(k\ge
k^{*}\)}$;\\
\item$ {\|\mu\| }(A_{k;t}\setminus E^{1})<\varepsilon
\quad\text{\rm for all \(\frac{1}{2}<t<1\) and all \(k\ge k^{*}\)}$;\\
\item $\|\mu\| (E^{1}\setminus  A_{k;t})< \varepsilon
\quad\text{\rm for  all \(\frac{1}{2}<t<1\) and all \(k\ge k^{*}\)}
$;\\
\item $\|\mu\|(E\setminus  A_{k;t})< \varepsilon \quad\text{\rm
for  all \(0<t<\frac{1}{2}  \) and all \(k\ge k^{*}\)}$.
\end{enumerate}
\end{lemma}

\begin{proof}
We first show (ii). With \(\frac{1}{2} <t <1\), choose \(
0<\varepsilon< t-\frac{1}{2} \).
 Let \(u_{k}\) denote (as usual) the mollification of
\(\charfn{E}  \). We know that \(u_{k}(y) \to u_{_{E}}(y) \) for
\(\H^{N-1}\)-a.e. \(y\) and therefore the same is true  for
\({\|\mu\|}\) as well. By Egorov's theorem, for any \(\varepsilon
>0\), there is  an open set \(U_{\varepsilon }\) such that
\({\|\mu\|}(U_{\varepsilon })<\varepsilon \) and that $|u_{k}(y) -
u_{_{E}}(y) |<\varepsilon$ for  all \(y \not\in U_{\varepsilon }\)
and for all \(k\ge k^{*}=k^{*}(\varepsilon, t)\).
 On \(A_{k;t}\setminus U_{\varepsilon }\), we have
\[
t <u_{k}(y).
\]
Since \(u_{k}(y)<u_{_{E }}(y)+\varepsilon \) on \(\rn \setminus
U_{\varepsilon }\), we have
\[
\frac{1}{2}  <t-\varepsilon <u_{_{E }}(y)\implies u_{_{E}}(y)=1
\implies y\in E^{1}.
\]
This yields
\[
A_{k;t}\setminus U_{\varepsilon }\subset E^{1} \implies
A_{k;t}\setminus E^{1}\subset U_{\varepsilon}.
\]
Since \({\|\mu\|} (U_{\varepsilon })<\varepsilon \), our desired
result (ii) follows.

For the proof of (i), given $0<t <\frac{1}{2}$, we choose $0<\ve <
t$ and proceed as above.

We next show (iv). With $0< t < \frac{1}{2}$, choose $0 < \ve
<\frac{1}{2}-t$. For all large \(k\), we have \(\abs{u_{k}(y) -
u_{_{E}}(y)}<\varepsilon \) for all \(y\not\in U_{\varepsilon }\).
Thus, on
 \(E \setminus U_{\varepsilon }\),
\[
\frac{1}{2}-u_{k}(y) \le u_{_{E}}(y)-u_{k}(y) < \varepsilon
\implies u_{k}(y) > \frac{1}{2} -\varepsilon > t \quad \mbox{for
all}\,\, y\in E \setminus U_{\varepsilon},
\]
which implies \(E \setminus U_{\varepsilon }\subset A_{k;t}.\)
Therefore, $E \setminus A_{k;t} \subset U_{\ve}$ and thus
${\|\mu\|}(E \setminus A_{k;t}) < \ve$.

For the proof of (iii), given $\frac{1}{2}<t < 1$, we choose $0 <\ve
< 1-t$ and proceed as above.
\end{proof}

\begin{corollary}\label{2results}
For each \(0<t<\frac{1}{2}\) and $\ve < \min \{t,\frac{1}{2}-t\}$,
there exists \(k^{*}=k^{*}(\varepsilon,t )>0\) such that
\begin{equation}\label{lessthanhalf}
{\|\mu\|}(A_{k;t} \Delta E)<\varepsilon \qquad \mbox{whenever}\,\,\,
k\ge k^{*}.
\end{equation}
For each \(\frac{1}{2}<t<1\) and $\ve < \min\{t-\frac{1}{2}, 1-t\}$,
there exists \(k^{*}=k^{*}(\varepsilon,t )>0\) such that
\begin{equation}\label{largerthanhalf}
{\|\mu\|}(A_{k;t}\Delta  E^{1})<\varepsilon
\qquad\mbox{whenever}\,\,\,
 k\ge k^{*}.
\end{equation}
\end{corollary}

\begin{remark}\label{closedopen} In the previous result, we used
open  superlevel sets \(A_{k;t}:= \{y : u_{k}(y)>t\}\). However, we
could have used closed superlevel sets \(\overline{A}_{k;t}:= \{y :
u_{k}(y)\ge t \}\) to obtain the same result. We also note that, for
an arbitrary Radon measure \(\omega\), we have
\begin{equation}\label{KO}
\omega (\overline{A}_{k;t})-\omega ({A}_{k;t}) =\omega(\partial
{A}_{k;t})=0
\end{equation}
for all but countably many \(t\), for the reason that the family of
sets \(\{\partial A_{k;t}:t\in \R\}\) is pairwise disjoint and any
Radon measure \(\omega\) can assign positive values to only a
countable number of such a family.
\end{remark}

\begin{corollary}\label{2resultsandabitmore}
For each  \(t>0\), there exist $\ve (t)$ and
\(k^{*}=k^{*}(\varepsilon,t)>0\) such that
\begin{enumerate}[\rm (i)]
\item ${\|\mu\|}(A_{k;t}\Delta E)
   ={\|\mu\|}(\bar{A}_{k;t}\Delta E)<\varepsilon$ for all but
   countably many \(t \in (0, \frac{1}{2})\)  and for \(k\ge k^{*}\);
\item ${\|\mu\|}(A_{k;t}\Delta E^{1})={\|\mu\|}(\bar{A}_{k;t}
   \Delta E^{1})<\varepsilon $ for all but countably many \(t \in
   (\frac{1}{2}, 1)\) and for \(k\ge k^{*}\);
\item ${\|\mu\|}(\partial A_{k;t}\Delta E)={\|\mu\|}(u_{k}^{-1}(t)
  \Delta E)<\varepsilon$ for almost all \(t \in (0, \frac{1}{2})\)
   and for \(k\ge k^{*}\);
\item ${\|\mu\|}(\partial A_{k,t}\Delta E^{1})={\|\mu\|}(u_{k}^{-1}(t)
 \Delta E^{1})<\varepsilon $ for almost all \(t\in (\frac{1}{2},1)\)
 and for \(k\ge k^{*}\).
\end{enumerate}
\end{corollary}

For the case $t=\frac{1}{2}$, only (i) and (iii) in Lemma
\ref{general} remain valid. To see this, we first show

\begin{lemma}
Let $\mu$ be a Radon measure on $\rn$ such that $\mu << \H^{N-1}$.
Let $E$ be a set of finite perimeter, and let $ u_k $ be the
mollification of $\chi_E$. Then, for $t=\frac{1}{2}$ and $\ve>0$,
there exists $k^* = k^* (\ve)$ such that
\begin{equation}
{\|\mu\|} (E^1 \setminus A_{k; \frac{1}{2}}) < \ve \quad \tn{ and
}\quad {\|\mu\|} (A_{k;\frac{1}{2}}\setminus E) < \ve.
\end{equation}
\end{lemma}

\begin{proof}
Since $u_k(y)\to u_E(y)$ for $\H^{N-1}$-a.e. $y$, the dominated
convergence theorem implies that $u_k \to u_E$ in
$L^1(\rn,{\|\mu\|})$. Thus, given any $\ve>0$, for $k$ large enough,
we have
\begin{equation}
\frac{\ve}{2}\geq \int_{\rn} \abs{u_E - u_k} d{\|\mu\|} \geq
\int_{E^1\setminus A_{k;\frac{1}{2}}} (u_E-u_k)d{\|\mu\|} \geq
(1-\frac{1}{2}) {\|\mu\|}(E^1\setminus A_{k;\frac{1}{2}}),
\end{equation}
which implies
\begin{equation}
{\|\mu\|} (E^1\setminus A_{k;\frac{1}{2}}) \leq \ve.
\end{equation}
In the same way, we compute
\begin{equation}
\frac{\ve}{2}\geq\int_{A_{k;\frac{1}{2}}\setminus E}\abs{u_E -
u_k}d{\|\mu\|} \geq
(\frac{1}{2}-0){\|\mu\|}(A_{k;\frac{1}{2}}\setminus E),
\end{equation}
which implies
\begin{equation}
{\|\mu\|}(A_{k;\frac{1}{2}}\setminus E) \leq \ve.
\end{equation}
\end{proof}

The following remark shows that, with $t=\frac{1}{2}$  and with
\(\mu = \H^{N-1} \rtangle  \partial^{*}E\ge 0 \), (ii) and (iv) in
Lemma \ref{general} do not hold.

\begin{remark}
If we define $ E:=\{ y\in\rn: \abs{y} \leq 1 \}$, then
$u_k^{-1}(\frac{1}{2})\subset \rn \setminus E$ for all $k$, and
therefore it is clear that
\begin{equation}
\H^{N-1}((A_{k;\frac{1}{2}} \setminus E^1) \cap \partial^*E) =
\H^{N-1}(\partial^*E) \nrightarrow 0 \quad\tn{ as } k \to \infty.
\end{equation}
If we now define $ E:= \{y\in \rn : \abs{y} \geq 1 \}$, then
$u_k^{-1}(\frac{1}{2})\subset E$ for all $k$ and thus
\begin{equation}
\H^{N-1}((E\setminus A_{k;\frac{1}{2}})\cap \partial^*E) =
\H^{N-1}(\partial^*E) \nrightarrow 0\quad \tn{ as } k \to \infty.
\end{equation}
\end{remark}

\begin{lemma}\label{slicesubded}
There exists a number \(0< C<\infty \) such that, for all positive
integers \(k\) and  almost all \(t\in (0,1)\),
\begin{equation}\label{sigbded}
\H^{N-1} (u_{k}^{-1}(t) )  \le C.
\end{equation}
\end{lemma}

\begin{proof}
{}From Corollary \ref{2results}, it follows that, for almost all
\(t\in(0,1)\), there exists a sequence of smoothly bounded sets
\(A_{k;t}\) such that either \(\charfn{ A_{k;t}}\to \charfn{E^{1}}\)
\({\mu }  \)-a.e. (if \(\frac{1}{2}<t<1\)), or \(\charfn{
A_{k;t}}\to \charfn{E^{}}\) \({ \mu } \)-a.e. (if
\(0<t<\frac{1}{2}\)). Since \(\mu = \H^{N-1} \rtangle
\partial^{*}E < < \H^{N-1}\), it follows that
\(\charfn{ A_{k;t}}\to \charfn{E}\) everywhere except for a set of
Lebesgue measure zero. Since these functions are integrable, we may
consider these functions as distributions and thus, by appealing to
Lebesgue's dominated convergence theorem, we may conclude
\[
\charfn{ A_{k;t}}\to \charfn{E}\quad\text{ in the sense of
distributions},
\]
and therefore
\[
\nabla \charfn{ A_{k;t}}\to \nabla \charfn{E}\quad\text{ in the
sense of distributions.}
\]
Since all the functions are in $BV$, we find that
\[
\nabla \charfn{ A_{k;t}}\, \stackrel{*}{\rightharpoonup} \, \nabla
\charfn{E} \qquad\text{in}\,\, \mathcal{M}(\rn),
\]
and the limit $\nabla \charfn{E}$ is independent of $t\in(0,1)$.
Therefore, the uniform boundedness theorem for measures (Corollary
\ref{unbMeasures}; also Lemma \ref{wkstconvergence}) implies that
these measures are uniformly bounded in $\mathcal{M}(\rn)$; that is,
there exists \(0<C<\infty \), independent of \(t\), such that, for
almost all \(t\in (0,1)\),
\[
\sup_{k}\|\nabla \charfn{ A_{k;t}}\|(\rn )\le C.
\]
Since
\[
\|\nabla \charfn{ A_{k;t}}\|=\H^{N-1} \rtangle u_{k}^{-1}(t),
\]
our result follows.
\end{proof}

\medskip
We now offer another proof of this result here for the purpose of
broadening the context of our development. By using the theory of
integral currents \cite{Fe}, the result is immediate. Rather than
actually introducing integral currents, we will introduce a small
structure that reflects the argument from integral currents. For
this, let \(\mathcal{V}\) denote the Banach space \(C^{1 }_{c}(K)\)
of vector fields \(\psi\) on \( K \) endowed with  norm
\[
\norm{\psi}:= \displaystyle{\sup_{y\in K}\big(\abs{
\psi(y)}+\sum_{i=1}^{N}|\nabla \psi_{i}(y)|\big)},
\]
where \(K\) is a compact set such that  \(E \Subset K\).

Let
\[
T_{_{E}}(\psi ):=\int_{E}\psi(y)\di y
\]
and, for almost every \(t\in (0,1)\), let
\[
T_{k;t}(\psi):= \int_{\{y\, :\, u_{k}(y)>t
\}}\psi(y)\,dy=\int_{A_{k;t}} \psi(y)\, dy \qquad\text{ for
\(\psi\in \mathcal{V}\)}.
\]
Then, for \(\psi\in \mathcal{V}\), we define the linear operators:
\[
\partial T_{E}(\psi):=
T_{E}(\div \psi)=\int_{E}\div \psi\, dy=-\int_{\partial^{*}E
}\psi\cdot \nnu \, d\H^{N-1},
\]
and
\[
\partial T_{k;t}(\psi):=
T_{k;t}(\div \psi)=\int_{A_{k;t}}\div \psi\, dy=-\int_{\po
A_{k;t}}\psi \cdot \nnu\, d\H^{N-1},
\]
where $\nnu$ is the interior unit normal.

Since \(u_{k}^{-1}( t)\) is a \(C^{\infty }\)-manifold, then
\[
\norm{ \partial T_{k;t}}:= \sup_{\norm{ \psi}\le1}\abs{\partial
T_{k;t}(\psi)}=\H^{N-1}( u_{k}^{-1}( t)).
\]
Indeed, with \(\psi := \displaystyle{ \frac{\nnu_{k}}{\abs{\nnu
_{k}}} }\) defined on the manifold \(u_{k}^{-1}( t)\), the
norm-preserving extension of \(\psi\) to all of \(\rn\) by Whitney's
extension theorem yields the inequality
\[
\norm{ \partial T_{k;t}}:= \sup_{\norm{ \psi}\le1}\abs{\partial
T_{k;t}(\psi)}\ge\H^{N-1}( u_{k}^{-1}( t)).
\]
The opposite inequality is obvious.

\medskip
Moreover, we find by the dominated convergence theorem that
\[
\lim_{k \to \infty }\po T_{k;t}(\psi) \to \po
T_{_{E}}(\psi)\,\qquad\mbox{for}\,\, \psi\in\mathcal{V},
\]
and therefore,
$$
\sup_{k}\{\abs{\partial
T_{k;t}(\psi)}\}<\infty\,\qquad\mbox{for}\,\, \psi\in\mathcal{V}.
$$
By the uniform boundedness principle (Theorem \ref{unbdprin}), we
see that, since \(\partial T_{k;t}\) is a linear functional on
\(\mathcal{V}\) whose week limit, $\po T_{_{E}}$, is independent of
$t$, we have
\[
\sup_{k}\H^{N-1}( u_{k}^{-1}( t))=\sup_{k}\norm{ \partial
T_{k;t}}<\infty,
\]
which gives our desired result.

\medskip
The above argument simply rephrases the following basic fact from
the theory of currents. We know that, since \(E\) has finite
perimeter, \(T_{_{E}}\) is an integral current. Moreover, the
currents \(T_{k;t}\) converge to \(T_{_{E}}\) weakly and therefore
so do their boundaries, \(\partial T_{k;t}\to \partial T_{_{E}}\);
that is,
\[
\int_{u_{k}^{-1}( t)}\sigma \di \H^{N-1}=\int_{\partial A_{k;t}}\sigma
\di \H^{N-1}\to\int_{\partial^{*}E }\sigma \di \H^{N-1}
\]
for each smooth differential $(N-1)$-form \(\sigma \).  Appealing to
Corollary \ref{unbMeasures} yields our result.

\begin{lemma}\label{dolordemanos}
  Let $u\colon \Omega \to \R$ be a Lipschitz function and \(E
  \Subset\Omega \) a set of finite perimeter.  Then
\[
\H^{N-1}(\partial^{*}E\cap u^{-1}(t))=0\quad\text{for almost all
\(t\)}.
\]
\end{lemma}

\begin{proof}
This result follows directly from Lemma \ref{transversal},
since we know that \(\H^{N-1} (\partial^{*}E )<\infty \) for any
bounded set of finite perimeter, \(E \subset \rn \).
\end{proof}

\begin{lemma} \label{tim}
For almost every $\frac{1}{2}<t <1$, we have
\begin{equation}\label{porfavor1}
 \H^{N-1} (\po^*E \cap u_k^{-1}(t)) =0
\end{equation}
and
\begin{equation}\label{porfavor2}
\lim_{k \to \infty} \H^{N-1} (\po^*E \cap A_{k;t}) =0.
\end{equation}
\end{lemma}

\begin{proof}
This can be seen as follows. If we use Corollary \ref{2results} with
$\mu=\H^{N-1}\rtangle\po^* E$, we obtain

$$ \lim_{k \to \infty}\mu (A_{k;t} \setminus E^1) =
  \lim_{k \to \infty} \H^{N-1}(A_{k;t} \cap \po^*E)=0.
$$
Clearly, (\ref{porfavor1}) follows from Lemma \ref{dolordemanos}
(see also Remark \ref{closedopen}).
\end{proof}

\begin{theorem}[Approximation theorem]\label{muydificil}
For almost every $\frac{1}{2}<t <1$, we have
$$
\lim_{k \to \infty} \H^{N-1} ((E^0 \cup \po^*E)\cap u_k^{-1}(t)) =0.
$$
\end{theorem}

\begin{proof}
Since the Lebesgue measure is absolutely continuous with respect to
\(\H^{N-1} \), then using \eqref{largerthanhalf} in Corollary
\ref{2results} with \(s> \frac{1}{2}\) leads to
\[
\abs{A_{k;s} \Delta E^{1} }\to  0 \qquad\text{as $k \to  \infty$}.
\]
Therefore, since
\[
R_{k;s}:=A_{k;s}\setminus E^{1}=A_{k;s}\cap (E^{0} \cup
\partial^{*}E^{0}),
\]
it follows
that
\begin{equation} \label{eq:5}
\abs{R_{k;s} } \to  0 \qquad\text{provided that \(s>\frac{1}{2}\)}.
\end{equation}
Remark \ref{closedopen} indicates that we have the option of
defining \(A_{k;t} := \{y : u_{k}(y)\ge t \}\) without altering the
development. With this option in force, we have \(u_{k}^{-1}(s)
\subset A_{k;s}\) and consequently, by the coarea formula (Theorem
\ref{coarea}),
\begin{align*}
\int_{R_{k;s} } \abs{\nabla u_{k}} \, dy
&= \int_{0}^{1}\H^{N-1} (u_{k}^{-1}(t) \cap R_{k;s}  )\di  t\\
&= \int_{s}^{1}\H^{N-1}\big(u_{k}^{-1}(t) \cap (E^{0}\cup
\partial^{*}E^{0})\big)\,\di  t.
\end{align*}

Since \(\nabla u_{k} \, \stackrel{*}{\rightharpoonup} \,\nabla
\charfn{E}\) and \(\norm{\nabla u_{k}}_{1}\le \norm{\nabla
\charfn{E}}\)\, (Lemma \ref{decreasing}), it follows from Vitali's
convergence theorem for \(s>\frac{1}{2}\) that
\[
\int_{R_{k;s}}\abs {\nabla u_{k}}\,dy \to 0.
\]
Thus, for a subsequence, we can conclude that, for a.e. $t>s$,
\[
\H^{N-1} \big(u_{k}^{-1}(t) \cap (E^{0}\cup
\partial^{*}E^{0})\big)\to 0\qquad\text{as \(k \to  \infty \)}.
\]
The dependence on the subsequence is illusory. The reason is that,
if there were a subsequence such that, for a.e. $t$,
\[
\H^{N-1}\big(u_{k}^{-1}(t) \cap (E^{0}\cup
\partial^{*}E^{0})\big)\to \alpha\ne 0\qquad\text{as \(k \to  \infty \)},
\]
then one could appeal to our previous argument to conclude that, for
some further subsequence and for a.e. $t$,
\[
\H^{N-1}\big(u_{k}^{-1}(t) \cap (E^{0}\cup
\partial^{*}E^{0})\big)\to 0\qquad\text{as \(k \to  \infty \)},
\]
which is contrary to our assertion that $\alpha\ne0$.

Since $s>\frac{1}{2}$ was fixed arbitrarily at the beginning of the
proof, we conclude that, for a.e. $t>\frac{1}{2}$,
\[
\H^{N-1}\big(u_{k}^{-1}(t) \cap (E^{0}\cup
\partial^{*}E^{0})\big)\to 0\qquad\text{as \(k \to  \infty \)}.
\]
\end{proof}

\section{Main Theorem}

In this section we establish our main result, Theorem \ref{main}.
Let \(\FF\in \DM^\infty_{loc}(\Omega)\).
We define a measure $\sigma_{k;t}$ for all Borel sets \(B\Subset
\Omega\) by
\begin{equation}\label{sigmaa}
\sigma_{k;t}(B):= \int_{ B \cap \po A_{k;t}}\FF(y)\cdot \nnu (y)\,
d\H^{N-1}(y),
\end{equation}
where $\FF(y)\cdot\nnu(y)$ denotes the normal trace of $\FF$ on the
smooth boundary $\po A_{k;t}$ introduced in Lemma \ref{general}.

We begin with a lemma that will lead to several of the assertions in
Theorem \ref{main}.

\begin{lemma}\label{connection}
If \(E \Subset\Omega\)  is an arbitrary set of finite perimeter,
then
 we have
\begin{equation}
 \int_{E}\FF \cdot \nabla u_{k}  \di y =\int_{0}^{1}\int_{E \cap
u^{-1}_k(t)}\FF \cdot
\nnu_{k}   \di  \H^{N-1} \di  t
\end{equation}
for any \(\FF\in L^{\infty}_{loc}(\Omega;\rn )\), where \(u_{k}\)
denotes the mollification of \(\chi_{_{E}} \) as introduced in
Definition {\rm \ref{mollifiers}} and Lemma {\rm \ref{mollifiers2}}.
\end{lemma}

\begin{proof} Let $\mathcal{N}$ be the set on which $\nabla u_k=0$.
Then
\begin{equation*}
 \begin{aligned}
\int_{E} \FF\cdot \nabla u_k \di y &=\int_{E \setminus \mathcal{N}}
\FF\cdot \nabla u_{k}
\di y+\int_{\mathcal{N} } \FF\cdot \nabla u_{k}\, dy \\
&= \int_{E \setminus \mathcal{N} } \abs{\nabla
u_k} \frac{\FF\cdot \nabla u_k}{\abs{\nabla u_k}} \di y+0\\
&=  \int_{E} \abs{\nabla u_k}g \di y,
\end{aligned}
\end{equation*}
where $g=\frac{\charfn{E \setminus \mathcal{N}}\FF\cdot \nabla
u_k}{\abs{\nabla u_k}}$. Then, by the Coarea Formula, we have
\begin{eqnarray*}
\int_{E} \FF\cdot \nabla u_k \di y &=&\int_{0}^{1}
\int_{u_k^{-1}(t)\cap (E\setminus \mathcal{N})} g \, d \H^{N-1} \di
t\label{coareaone}
\\
&=&\int_{0}^{1} \int_{u_{k}^{-1}(t) \cap (E\setminus \mathcal{N})}
\FF \cdot \nnu_{k}\,
d\H^{N-1}\, dt \nonumber\\
&=&\int_{0} ^{1}  \int_{u_k^{-1}(t) \cap E} \FF\cdot \nnu_{k} \, \di
\H^{N-1} \di t, \nonumber
\end{eqnarray*}
where we used \(\nnu_{k} (y)=\frac{\nabla u_{k}(y)}{\abs{\nabla
u_{k}(y)}}\) for \(y\in u_{k}^{-1}(t)\cap (E \setminus \mathcal{N}
)\).
\end{proof}

With the help of Theorem
\ref{GGforC1} and the results in \S 4,
we now establish our main theorem.

\begin{theorem}[Main theorem]\label{main}
Let \(\Omega\subset\rn\) be an open set. Suppose that $\FF \in
\DM^\infty_{loc}(\Omega)$ with \(\div \FF=\mu\in
\mathcal{M}(\Omega)\). Let $E\Subset\Omega$ be a set of finite
perimeter. Then
\begin{enumerate}[\rm (I)]
\item For
 almost every
\(s\in (\frac{1}{2},1)\), there exist a signed measure \(\sigma_{i}
\) (independent of $s$) and a family of sets \(A_{k;s}\) with smooth
boundaries such that
\begin{enumerate}[\rm\bf(a)]

\item
  \begin{math}
    \Vert\mu\Vert (A_{k;s}\Delta E^{1})  \to 0;
  \end{math}

\item The measure \(\sigma_{i}\) is the weak* limit of the measures
\(\sigma_{k;s}\);

\item\text{\( \sigma_{i}   \) is carried by \(\partial ^{*}E \)} in
the
  sense that \(\norm{ \sigma_{i}}(\Omega\setminus\partial ^{*}E)=0\);

\item \(\norm{ \sigma_{i}}< <\H^{N-1}\rtangle \po^* E \);

\item \(\displaystyle{\lim_{ k\to \infty }\H^{N-1}(\partial
A_{k;s}\cap (E^{0}\cup \po^*E))=0}\);

\item \(\displaystyle {\lim_{ k \to \infty }\|\sigma_{k;s}\|(
E^{0}\cup \po^*E)=0}\);

\item The density of \(\sigma_{i}  \), denoted as
\(\mathscr{F}_{i}\cdot \nnu\), is called
the {\bf interior normal trace relative to \(E\)} of \(\FF \)  on
\(\partial ^{*}E \) and satisfies
\begin{equation}\label{it}
\int_{E^{1}} \div \FF =:\mu(E^{1})=-\sigma_{i} (\partial ^{*}E)= -
\int_{\partial^{*}E} (\mathscr{F}_{i} \cdot \nnu)(y)\, d\H^{N-1}(y);
\end{equation}

\item If \((2\FF \cdot \nabla u_{k})\charfn{E}  \) is considered as a
sequence of  measures,
then this sequence converges weak* to the measure
  \((\mathscr{F}_{i} \cdot \nnu)\H^{N-1}\rtangle \partial^{*}E\), i.e.,
\[
(2\FF \cdot \nabla u_{k})\charfn{E} \, \stackrel{*}{\rightharpoonup}
\, (\mathscr{F}_{i}\cdot\nnu)\H^{N-1}\rtangle \partial^{*}E
  \qquad\mbox{in}\,\, \mathcal{M}(\Omega);
\]
\item \(\norm{\sigma_{i}}=\norm{ \mathscr{F}_{i}\cdot
\nnu}_{\infty;\partial^{*}E,\H^{N-1}} \le \norm{F}_{\infty; E}\).
\end{enumerate}

\item For
 almost every
\(s\in (0,\frac{1}{2})\), there exist a signed measure \(\sigma_{e}
\) (independent of $s$) and a family of sets \(A_{k;s}\) with smooth
boundaries such that
\begin{enumerate}[\rm\bf(a)]
\item
  \begin{math}
    \Vert\mu\Vert (A_{k;s}\Delta E) \to 0;
  \end{math}

\item The measure \(\sigma_{e}\) is the weak* limit of
\(\sigma_{k;s}\);

\item\text{\(\sigma_{e}\) is carried by \(\partial ^{*}E \)} in
the  sense that \(\norm{ \sigma_{e}}(\Omega\setminus\partial
^{*}E)=0\);

\item \(\|\sigma_{e}\|< <\H^{N-1}\rtangle \po^* E\);

\item \(\displaystyle\lim_{ k\to \infty }\H^{N-1}(\partial
A_{k;s}\cap E) =\lim_{ k\to \infty }\H^{N-1}(u_{k}^{-1}(s) \cap
E)=0\);

\item \(\displaystyle{ \lim_{ k \to \infty }\|\sigma _{k;s}\|(\Omega
\setminus E^{0})= \lim_{ k \to \infty }\|\sigma_{k;s}\|( E)=0}\);

\item The density of \(\sigma_{e}  \), denoted as \(\mathscr{F}_{e}
\cdot \nnu \), is called
the {\bf exterior normal trace relative to \(E\)} of \(\FF \) on
\(\partial ^{*}E \) and satisfies
\begin{equation}\label{ot}
\int_{E} \div \FF =:\mu(E)=-\sigma_{e} (\partial ^{*}E)= -
 \int_{\partial^{*}E} (\mathscr{F}_{e}\cdot \nnu)(y)\, d\H^{N-1}(y);
\end{equation}

\item If \( (2\FF \cdot \nabla u_{k})\charfn{E^0}  \) is considered as a
sequence
of  measures,
then this sequence converges weak* to the measure
 \( (\mathscr{\F}_{e} \cdot \nnu)\H^{N-1}\rtangle \partial^{*}E\), i.e.,
\[
(2\FF \cdot \nabla u_{k})\charfn{E^0}\,
\stackrel{*}{\rightharpoonup} \, (\mathscr{F}_{e}\cdot
\nnu)\H^{N-1}\rtangle \partial^{*}E \qquad\mbox{in}\,\,
\mathcal{M}(\Omega);
\]

\item \(\norm{\sigma_{e}}=\norm{ \mathscr{F}_{e}\cdot
\nnu}_{\infty;\partial^{*}E,\H^{N-1}} \le
\norm{\FF}_{\infty;\Omega\setminus E}\).
\end{enumerate}
\end{enumerate}
\end{theorem}

\begin{proof}
We will prove only part (I), since the proof of part (II) is
virtually identical. For notational simplicity,  we will use the
notation $\sigma$ rather than $\sigma_{i}$ in the proof of part (I).
Throughout the proof, we will consider only those values of \(s\in
(\frac{1}{2},1)\) for which the results in \S 4 are valid for all
the mollified functions, \(u_{k}\), of \(\charfn{E} \) thus omitting
at most a set of measure zero. Without increasing the measure of
this exceptional set, call it $\mathcal{S}$, we will also include
those values of \(s\) for which \(\Vert\mu\Vert (u_{k}^{-1}(s)) \ne
0 \) for all \(k\). For the rest of the proof, we fix such an
$s\notin\mathcal{S}$.

\medskip
We start with {\bf (a)}. We consider the sets $A_{k;s}$ as in Lemma
\ref{general}. The desired result  follows directly from Corollary
\ref {2results}.

\medskip
{\bf (b) } Since $\FF$ is bounded, Lemma \ref{slicesubded} implies
that there exists a constant $C$ such that
\begin{equation}
\norm{\sigma_{k;s}}(\Omega)\le C,
\end{equation}
which yields, as in \eqref{sigmaconverge}, the existence of a signed
measure
\(\sigma _{s}\) such that
\begin{equation} \label{yayaya}
\sigma _{k;s}\, \stackrel{*}{\rightharpoonup} \, \sigma_{s}
\quad\mbox{in}\,\,\,\mathcal{M}(\Omega).
\end{equation}
Utilizing \eqref{largerthanhalf}, we also obtain that
$\mu(A_{k;s})\to \mu(E^{1})$. Since Theorem \ref{GGforC1} yields
$\mu(A_{k;s})=-\sigma_{k;s}(\Omega)$, we obtain, after letting
$k\to\infty$,
\begin{equation}\label{indep}
\mu(E^{1})=-\sigma_{s}(\Omega).
\end{equation}
Since the left side of equation (\ref{indep}) is independent of $s$,
we show next that \(\sigma _{s}\) is also independent of \(s\) (and
independent of the sequence in the weak* convergence
(\ref{yayaya})). To see this, we fix any $\phi\in C_c^1(\Omega)$ and
note that, since $\FF$ is a divergence-measure field, the product
rule in Lemma \ref{productrule} implies that $\phi\FF$ is also a
divergence-measure field. Proceeding as above with $\phi\FF$ instead
of $\FF$, we obtain
 \begin{equation}
   \int_{E^1} \div (\phi \FF) = -\int_{\Omega}\phi\, d\sigma_{s}
 \end{equation}
for any $\phi \in C_c^1(\Omega)$. Therefore, for any two measures
$\sigma_{s}$ and $\sigma_{s'}$ that are the limits in
(\ref{yayaya}), we have that  $\int_{\Omega}\phi \,
d\sigma_{s}=\int_{\Omega}\phi\, d\sigma_{s'}$ for any $\phi \in
C_c^1(\Omega)$ and thus we conclude that $\sigma_{s}= \sigma_{s'}$.

\medskip
{\bf (c) } Let \( A\subset\Omega\setminus \partial^{*}E \) be an
arbitrary Borel set. Referring to \eqref{normHaus}, we see that
\begin{equation*}
\norm{ \nabla \charfn{E}}(A)=0.
\end{equation*}
On the other hand,
we know
\begin{equation}
  \label{outerreg}
  \begin{aligned}
0
&=\Vert\nabla \charfn{E}\Vert(A)\\
&=\inf\{\Vert\nabla \charfn{E} \Vert
(U):\text{\(A\subset U,\, U \)\;open}\}\\
&=\inf\{\Vert\nabla \charfn{E} \Vert
(U):\text{\(A\subset U,\, U \)\;open}, \norm{\nabla \charfn{E}}(
\partial U)=0\}.
\end{aligned}
\end{equation}
In order to prove that $\norm{\sigma }(A)=0$, we proceed by
contradiction via  assuming   \(\norm{\sigma }(A) >0 \). {}From
\eqref{outerreg}, there is an open set \(U\supset A \) such that
\(\norm{\nabla \charfn{E}}( \partial U)=0\) and
%
%
\begin{equation}
  \label{squeeze}
   \norm{ \nabla
\charfn{E}}(U)
   < \frac{\Vert \sigma   \Vert (A)}{ 2\Vert \FF\Vert_{\infty}}.
\end{equation}
{}From Lemma \ref{connection}, we have
\begin{equation}
 \int_{U} \abs{\FF \cdot \nabla u_{k}}\, dy =\int_{0}^{1}\int_{U \cap
u^{-1}_k(t)}\abs{\FF \cdot
\nnu_{k} }  \di  \H^{N-1} \di  t.
\end{equation}
Since \(U\) is open and  \(\sigma_{k;t}\,
\stackrel{*}{\rightharpoonup} \,\sigma \) in $\mathcal{M}(\Omega)$,
\begin{eqnarray*}
\norm{\sigma } (A)
&\le& 2\int_{\frac{1}{2}}^{1}\norm{\sigma } (U) \di t
\leq 2\int_{\frac{1}{2}}^{1}\liminf_{k\to\infty}\norm{\sigma_{k;t}}(U)\di t\\
&\leq& 2\int_{0}^{1}\liminf_{k\to\infty}\norm{\sigma_{k;t}}(U)\di t
\leq 2\liminf_{k\to\infty}\int_{0}^{1}\norm{\sigma_{k;t}}(U)\di t,
\end{eqnarray*}
by Fatou's lemma. Therefore, we have
\begin{eqnarray*}
\norm{\sigma } (A)&\le
&2\liminf_{k\to\infty}\int_{0}^{1}\int_{u_{k}^{-1}(t) \cap U}\abs
{\FF(y)\cdot \nnu(y)} \di \H^{N-1}(y) \di  t
\\
&= & 2\liminf_{ k \to \infty }\int_{U} \abs{\FF \cdot \nabla u_{k}}\,
dy\\
&\leq& 2\norm{\FF} _{\infty }\lim_{k \to \infty}\int_{U} \abs{ \nabla
u_k}\, dy\\
&= & 2\norm{\FF} _{\infty }\norm{\nabla \charfn{E} }(U)\\
&< & \Vert \sigma  \Vert(A),
\end{eqnarray*}
where we used Lemma \ref{mollifiers2}(iv) and the fact that \(\norm{
\nabla \charfn{E}}(\partial  U)=0 \). This yields a contradiction
and thus establishes our result.

\medskip
{\bf (d) } Let \(A \subset \Omega\) be a Borel set with \(\H^{N-1}
(A)=0\). Then, appealing to  \eqref{normHaus}, we find that
\(\norm{\nabla \charfn{E} }(\Omega)=0\). From this, the proof can be
proceeded precisely as in {\bf (c)} to yield our desired conclusion.
%
%

\medskip
{\bf (e) } This is the result of Theorem \ref{muydificil}.

\medskip
{\bf (f) } In view of the definition
\[
\sigma _{k;s}(B):= \int_{\partial A_{k;s}\cap B}\FF(y)\cdot \nnu(y)\, d
\H^{N-1} (y),
\]
and the fact that the normal trace is bounded, the result follows
immediately from {\bf(e)}.

{\bf (g) } From {\bf (a)}, we have the existence of smoothly bounded
sets such that
\begin{equation}\label{good1}
\Vert\mu\Vert (A_{k;s}\Delta E^{1})\to0 \qquad\mbox{as}\,\,\,
k\to\infty,
\end{equation}
where $s>1/2$ was fixed at the beginning of the proof. From Theorem
\ref{GGforC1}, we know that our desired result holds for smoothly
bounded sets:
\begin{equation}\label{antes1}
\mu (A_{k;s}):_{}=\int_{ A_{k;s}}\div \FF=-\int_{\partial
A_{k;s}}\FF(y)\cdot\nnu (y)\, d\H^{N-1}(y).
\end{equation}
We note that, with our notation in force, we may write
\eqref{antes1} as
\begin{equation}
  \label{short}
\mu(A_{k;s})=-\sigma_{k;s}(\Omega)=-\sigma_{k;s} (\partial A_{k;s}).
\end{equation}

Since
\begin{equation} \label{needforh}
\mu(A_{k;s})\to \mu (E^{1})\quad\text{and}\quad \sigma
_{k;s}(\Omega)\to \sigma(\Omega) \qquad\,\mbox{as}\,\,\, k\to\infty,
\end{equation}
we obtain
\[
\mu (E^{1})=-\sigma (\partial^{*}E ).
\]
Because \(\norm{\sigma} < < \H^{N-1} \rtangle \po^* E \), we know
that there exists \(\mathscr{F}_i \cdot \nnu \in L^{1}(\partial^{*}E
)\) such that
\[
\sigma(B)=\int_{B\cap \partial^{*}E }(\mathscr{F}_i \cdot \nnu)( y)\di
\H^{N-1}(y),
\]
which gives \eqref{it}.

\medskip
{\bf(h)} From Lemma \ref{connection}, we obtain
\[
\lim_{k \to \infty } \int_{E} \FF\cdot \nabla u_k \,\di y =\lim_{k\to
\infty}
 \int_{0 }^{1}  \int_{u_k^{-1}(t)\cap E } \FF(y)\cdot \nnu_{k} (y) \di
\H^{N-1}(y)
\di t=
\lim_{k \to
  \infty } \int _{0 }^{1} \sigma_{k;t}(E)\di t.
\]
Thus,
\begin{align*}
 \overline{\charfn{E}\FF \cdot \nabla u_E}(\Omega)
&:=  \lim_{k \to \infty } \int_{\Omega}\charfn{E}  \FF\cdot \nabla u_k
\di y\\
&=  \lim_{k \to \infty } \int_{E} \FF\cdot \nabla u_k \di y\\
&=\lim_{k \to \infty }\int _{\frac{1}{2} }^{1} \sigma_{k;t} (E) \di
t+\lim_{k \to \infty }\int _{\frac{1}{2} }^{1} \sigma_{k;t} (E^{0})
\di t\quad\text{(by {\bf (f)} above)}\\
&=\lim_{k \to \infty }\int _{\frac{1}{2} }^{1} \sigma_{k;t} (\Omega)
\di t\\
&= \frac{1}{2}\sigma(\Omega) \,.
\end{align*}
Let $\varphi$ be a function in $C_c^1(\Omega)$. Since $\varphi \FF$
is also a bounded divergence-measure field, we can proceed as above
with the vector field $\varphi \FF$ instead of $\FF$ to conclude
that
\begin{equation}
\int_{\Omega}\,\,\varphi\,  d\,\overline{\charfn{E}\FF \cdot \nabla
u_E} = \frac{1}{2}\int_{\Omega} \varphi \,\,d \sigma,
\end{equation}
which implies that $\sigma =2\overline{\charfn{E}\FF \cdot \nabla
u_E}$.

\medskip
{\bf(i) }
We have that, for $\H^{N-1}$-a.e. $y\in \po^*E$,
\begin{equation} \label{cansada}
\mathscr{F}_i\cdot \nnu (y) =\lim_{r \to 0}
\frac{\sigma(B(y,r))}{\norm{\nabla \chi_E}(B(y,r))},
\end{equation}
where we can choose the balls $B(y,r)$ such that $\norm{\nabla
\chi_E}(\partial B(y,r))=\norm{\sigma}(\partial B(y,r))=0$.

Using a similar argument as in {\bf{(c)}}, we obtain
\begin{align*}
\norm{\sigma}(B(y,r))
&=2\int_{\frac{1}{2}}^{1}\norm{\sigma}(B(y,r))\, dt\\
&= 2\lim_{k \to
\infty}\int_{\frac{1}{2}}^{1}\norm{\sigma_{k;t}}(B(y,r))\, dt\\
&= 2\lim_{k \to
\infty}\int_{\frac{1}{2}}^{1}\big(\norm{\sigma_{k;t}}(B(y,r) \cap
E^1) +
     \norm{\sigma_{k;t}}(B(y,r) \cap (E^0 \cup \po^*E))\big)\, dt\\
&= 2\lim_{k \to
\infty}\int_{\frac{1}{2}}^{1}\norm{\sigma_{k;t}}(B(y,r) \cap E^1)\,
dt\\
&= 2\lim_{k \to \infty}\int_{\frac{1}{2}}^{1} \int_{u_k^{-1}(t)\cap
B(y,r)\cap E^1}\abs{\FF\cdot\nnu}d\H^{N-1}dt\\
&\leq 2\norm{\FF}_{\infty;E^1} \lim_{k \to
\infty}\int_{\frac{1}{2}}^{1} \int_{u_k^{-1}(t)\cap B(y,r)}
d\H^{N-1}dt,
\end{align*}
where we have used the fact that $\norm{\sigma_{k;t}}(E^0 \cup
\po^*E) \to 0$ as $k\to\infty$ for a.e. $t>1/2$.

Therefore, from \eqref{cansada}, we obtain
\begin{align*}
\abs{\mathscr{F}_i \cdot \nnu(y)} &\leq\lim_{r \to 0}
\frac{\norm{\sigma}(B(y,r))}{\norm{\nabla \chi_E}(B(y,r))}\\
&\leq 2\norm{\FF}_{\infty;E^1} \lim_{r \to 0}\lim_{k \to
\infty}\frac{\int_{\frac{1}{2}}^{1} \int_{u_k^{-1}(t)\cap
B(y,r)} d\H^{N-1}dt}{\int_{B(y,r)}\abs{\nabla u_k}}\\
&= 2\norm{\FF}_{\infty;E^1} \lim_{r \to 0}\lim_{k \to
\infty}\frac{\int_{\frac{1}{2}}^{1} \int_{u_k^{-1}(t)\cap B(y,r)}
d\H^{N-1}dt}{\int_{0}^{1} \int_{u_k^{-1}(t)\cap
B(x,r)} d\H^{N-1}dt}\\
&=\norm{\FF}_{\infty;E^1}= \norm{\FF}_{\infty;E}. \qedhere
\end{align*}
\end{proof}
As a direct result, we obtain the Gauss-Green theorem for
divergence-measure fields over sets of finite perimeter. This also
shows that our definition of the normal trace is in agreement with
that given in the sense of distributions, Definition \ref{dmfield}.

\begin{theorem}[Gauss-Green theorem]\label{generalGG}
\label{above} Let \(\Omega\subset\rn\) be an open set. Let $\FF
\in\DM^\infty_{loc}(\Omega)$ and
let $E \Subset\Omega$ be a bounded set of finite perimeter. Then,
\begin{equation}
\int_{E^1}\varphi\, \div \FF + \int_{E^1} \FF \cdot \nabla \varphi =-
\int_{\partial^* E} \varphi\, (\mathscr{F}_{i} \cdot \nnu)\, d
\H^{N-1}
\end{equation}
for all $\varphi \in C^{\infty}_{c}(\Omega  )$, where $\nnu$ is the
measure-theoretic interior unit normal to $E$ on $\po^*E$.
\end{theorem}

\begin{proof}
{}From Lemma \ref{productrule}, it follows that $ \varphi \FF $ is a
bounded divergence-measure field and
\begin{equation}
\label{patricio} \div (\varphi \FF) = \varphi\, \div \FF + \FF \cdot
\nabla \varphi.
\end{equation}
Following the proof of Theorem \ref{main} applied to $\varphi \FF$
(instead of $\FF$), (see Theorem \ref{productrule}), we obtain
\begin{equation}
\int_{E^1} \div (\varphi \FF) =-\int_{\partial^* E} \varphi\,
(\mathscr{F}_{i} \cdot \nnu)\, d \H^{N-1},
\end{equation}
which, due to (\ref{patricio}), gives the desired result.
\end{proof}

We conclude this section with the following remark.

\begin{remark} \label{cor:5.4}
Theorem {\rm \ref{above}} implies that, when $E$ is an open set of
finite perimeter, our trace $ \mathscr{F}_{i} \cdot \nnu $ agrees
with the one defined in {\rm (\ref{tracedisplay})}.
\end{remark}

\section{The Divergence-Measures of Jump Sets via the Normal Traces}

In Theorem \ref{main}, we have defined the interior and exterior
normal traces of $\FF \in \DM^{\infty}_{loc}(\Omega)$,
$\mathscr{F}_{i}\cdot \nnu$ and $\mathscr{F}_{e}\cdot \nnu$, over
the boundary of a set of finite perimeter $E \Subset\Omega$. In
order to obtain the interior normal trace of $\FF$ on
$\po^*\tilde{E}$, where $\tilde{E}:=E^0 \cup \po^m E$, we reproduce
the proof of Theorem \ref{main} and apply it to $\tilde{E}$.
Therefore, the trace measure, denoted by $\sigma_{-}$, is obtained
by using the level sets $B_{k;s}=\{v_k>s\}$ for some $s\in
(\frac{1}{2},1)$, where $v_k$ is the mollification of
$\charfn{\tilde{E}}$. We note that, for all $y\in \Omega$,
$$
\rho_{\ve}*\charfn{E}(y)+ \rho_{\ve}*\charfn{\tilde{E}}(y) =1,
$$
and therefore
$$
v_k^{-1}(s)= u_k^{-1}(1-s),
$$
where $1-s \in (0,\frac{1}{2})$. Since $-\nnu$ is the interior unit
normal to $\tilde{E}$, we have
\begin{eqnarray}
\label{compadecete} \sigma_-(\Omega)&=&-\lim_{k \to
\infty}\int_{\partial B_{k;s}}\FF\cdot\nnu\, d\mathcal{H}^{N-1}=
                -\lim_{k \to \infty}\int_{\partial
A_{k;1-s}}\FF\cdot\nnu\,d\mathcal{H}^{N-1}\\
&=&-\lim_{k \to \infty} \sigma_{k;1-s}(\rn)= -\sigma_e
(\Omega).\nonumber
\end{eqnarray}
The following observation now becomes  evident.

\begin{corollary}
\label{intext} The interior trace of \(\FF\) relative to \(\tilde{E}
\) on \(\partial^{*}E \) is the same
 as minus the exterior trace of \(\FF\) relative to \( E\) on
\(\partial^{*}E \).
\end{corollary}

In order to establish the relation between $\sigma_i$ and
$\sigma_e$, we subtract \eqref{ot} from \eqref{it} and obtain the
following formula for $\mu=\div \FF$.

\begin{corollary}
\label{jump} $\quad$  $\displaystyle{ \mu (\partial^{*}E)=
  \int_{\partial^{*}E }(\mathscr{F}_{i}\cdot \nnu-\mathscr{F}_{e}\cdot
\nnu)(y)\di \H^{N-1}( y)} $.
\end{corollary}

We offer the following simple example to illustrate our result. This
example also dramatically demonstrates the difference between the
classical derivative and the weak (distributional) derivative.

\begin{example}
\noindent {\rm Consider the most elementary situation:  \(N=1\),
\(\Omega :=  (-1, 2)\), \(E := [0,1]\), and \(f\) is a
non-decreasing function defined on \((-1,2)\) which is continuous
everywhere except at \(y=0,1\), at which points we assume that $f$
has right-continuity.

\medskip
(i)  Case \(\frac{1}{2}<s<1\). Since \(f\) is in $BV$, we know that
\(f'=\mu \) for some measure \(\mu\). Then, according to Theorem
\ref {main},
\[
\mu (E^{1}) =\mu ((0,1)):=\int_{0+}^{1-}f' =\mathscr{F}_{i}\cdot
\nnu(1)-\mathscr{F}_{i}\cdot \nnu(0),
\]
where \(\mathscr{F}_{i}\cdot \nnu(1)=\lim_{y\to 1-}f(y)\) and
\(\mathscr{F}_{i}\cdot \nnu(0)=f(0+) \). Indeed, the sets
\(A_{k;s}\), with fixed  \(\frac{1}{2}<s<1 \), form a nested family
of open intervals contained in \([0,1]\). The measures \(\sigma_{k}
\) correspond to \(f\) evaluated on the point masses located at
\(y_{k}\);  thus, as in \eqref{needforh}, \(f(y_{k})\) converges to
a limit, \(\mathscr{F}_{i}\cdot \nnu(1)\).

\medskip
(ii)  Case \(0<s<\frac{1}{2}\). Then the sets \(A_{k;s}\),
\(0<s<\frac{1}{2} \), form a nested family of open intervals
containing \([0,1]\). Similar to the above, we have
\[
\mu (E) =\mu ([0,1]):=  \int_{0-}^{1+}f' =\mathscr{F}_{e}\cdot
\nnu(1)-\mathscr{F}_{e}\cdot \nnu(0),
\]
the measures \(\sigma_{k} \) correspond to \(f\) evaluated on the
point masses located at \(y_{k}\), and thus \(\mathscr{F}_{e}\cdot
\nnu(1)=f(1+)\) and \(\mathscr{F}_{e}\cdot \nnu(0)=\lim_{y\to
0-}f(y)\). }
\end{example}

\section{Consistency of the Normal Traces with the Classical Traces}

We now proceed to show the consistency of our normal traces with the
classical traces when $\FF$ is continuous. First we have the following
lemma.

\begin{lemma}\label{consistency}
\label{nolocreo} Let $\mu=\div \FF$ for $\FF \in
\mathcal{DM}^{\infty}(\rn)\cap C(\rn;\rn)$. Then
\[
\norm{\mu} (G)= 0
\]
for any set $G$ that can be written as the graph of a Lipschitz
function $f$.
\end{lemma}

\begin{proof}
First we have
$$
G:=\{(y',f(y'))\, :\,  y' \in W\subset \mathbb{R}^{N-1} \}.
$$

By regularity of $\mu$, it suffices to show that $\mu(K)=0$ for any
compact set $K \subset G$. Given any compact set $K \subset G$, let
$U_k \subset \rn$ be a sequence of open sets satisfying
\begin{equation}\label{condicion}
\mu(U_k \cap G) \to \mu(K).
\end{equation}
Fix any set $U_k$. We note by Besicovitch's theorem
that $U_k$ can
be written up  to a set of $\norm{\mu}$-measure zero as a countable
union of disjoint open parallelepipeds $I_i^k$ (the fact that we can
use parallelepipeds instead of balls follows from Morse
\cite{apmorse}). Thus, we have
\begin{equation} \label{queserademi}
\medcup_{i=1}^{\infty}I_i^k \subset U_k\quad  \tn{ and}\quad
\norm{\mu}(U_k \setminus \medcup_{i=1}^{\infty}I_i^k)=0.
\end{equation}
Denote $U_k$ simply as $U$ and $I_i^k$ as $I_i$. We fix an $i$ and
note that, for $t$ small enough, the graphs $T_t:=\{(y',
f(y')+t)\,:\, y'\in W\subset \R^{N-1}\}$ and $B_t:=\{(y',
f(y')-t)\,:\, y'\in W\subset \R^{N-1}\}$ are contained in $I_i$. Let
$R_t$ be the region inside $I_i$, bounded above and below by $T_t$
and $B_t$ respectively. For a.e. $t$, we define
$$
\alpha_t= \int_{\po R_t \setminus (T_t \cup B_t)} \FF(y) \cdot
\nnu(y)\, d\H^{N-1}(y)\qquad \text{for a.e. \(t\)},
$$
where $\nnu(y)$ is the interior unit normal to $R_t$ on  $\partial
R_t\setminus(T_t\cup B_t)$. Since Lemma \ref{Fuglede} applies to
$R_t$ for a.e. $t$, we arrive at
\begin{eqnarray}
\mu(R_t)&=&\int_{R_t} \div \FF \nonumber\\
&=&-\int_{B_{t}} \FF(y) \cdot \nnu(y)\, d\H^{N-1}(y) -
\int_{T_{t}} \FF(y) \cdot \nnu(y) \, d\H^{N-1}(y) - \alpha_t \nonumber\\
& =&\int_{T_t}\FF(y', y_n-2t) \cdot \nnu (y)\, d\H^{N-1}(y)
-\int_{T_t}\FF(y',y_n) \cdot \nnu (y)\, d\H^{N-1}(y)
- \alpha_t \nonumber \\
&= & \int_{T_t} (\FF(y', y_n-2t)-\FF(y', y_n))
          \cdot \nnu (y)\, d\H^{N-1}(y) - \alpha_t. \nonumber
\end{eqnarray}
Since $\FF$ is continuous and $\alpha_t \to 0$ as $t \to 0$, we find
that there exists $t_0(\ve,F,G)>0$ such that
$$
\mu(R_t) \leq \ve \qquad \text{for all} \,\, t \leq t_0(\ve,F,G).
$$
Then we have
$$
\mu( I_i \cap G) = \lim_{t \to 0} \mu(R_t) \leq \ve,
$$
which implies $\mu(I_i \cap G)=0$ since $\ve$ is arbitrary.
Therefore, using (\ref{queserademi}), we obtain
$$
 \mu(K)= \lim_{k \to \infty} \mu(U_k \cap G) =
 \lim_{k \to \infty} \sum
\mu(I_i^k \cap G)=0.
$$
\end{proof}

\begin{theorem} If $\FF \in \DM^{\infty}_{loc}(\Omega)$ is
continuous and $E\Subset\Omega$ is a set of finite perimeter, then
$\sigma_i =\overline{F \cdot \nabla u_E}$, and the normal trace
$\mathscr{F}_i\cdot \nnu$ is in fact the classical dot product
$\FF\cdot \nnu$, where $\nnu$ is the interior unit normal to $E$.
\end{theorem}

\begin{proof}
We recall that, by definition, $E=E^1 \cup \po^* E$. Denote $\tilde
E= E^0 \cup \po^* E$. Then we have
\begin{eqnarray*}
\overline{\FF \cdot \nabla u_E}&= &\lim_{k \to \infty}
    \int_{\rn}  \FF \cdot \nabla u_k\, dy\\
   &= &  \lim_{k \to \infty}
    \int_{\rn} \chi_E \FF \cdot \nabla u_k\, dy +
\lim_{k \to \infty} \int_{\rn} \chi_{\tilde E} \FF \cdot \nabla u_k\,
dy.
\end{eqnarray*}
If $v_k$ denotes the convolution $\chi_{\tilde E} * \rho_{1/k}$,
since $u_k+ v_k=1$, we obtain
\begin{eqnarray}\label{}
\overline{\FF \cdot \nabla u_E}&= &\lim_{k \to \infty}
    \int_{\rn} \chi_E \FF \cdot \nabla u_k\, dy -
\lim_{k \to \infty} \int_{\rn} \chi_{\tilde E} \FF \cdot \nabla v_k\, dy
\nonumber\\
&=& \frac{\sigma_i}{2}+ \frac{\sigma_e}{2}
= \frac{1}{2}(\sigma_i + \sigma_i - \mu(\po^* E))\nonumber \\
&=& \sigma_i - \frac{1}{2}\mu(\po^* E),\nonumber
\end{eqnarray}
where we used Theorem \ref{main} {\bf{(h)}} and Corollary
\ref{jump}.

Since $\po^* E$ is an $(N-1)$-rectifiable set (see
(\ref{almostsmooth})), it follows from Lemma \ref{nolocreo} that
$$
\norm{\mu}(\po^* E)=0,
$$
that is,
$$
\overline{\FF \cdot \nabla u_E}= \sigma_i.
$$
Thus, for
${\mathcal{H}}^{N-1}$-a.e. $y\in \po^* E$,
\begin{eqnarray*}
(\mathscr{F}_i \cdot \nnu)(y) &=&\lim_{r\to 0}\frac{\overline{
\FF\cdot\nabla u_E}(B(y,r))}{\norm{\nabla\chi_E}(B(y,r))} =
\lim_{r\to 0}\lim_{k \to \infty}
     \frac{\int_{B(y,r)}\FF \cdot \nabla u_k\, dx}{\int_{B(y,r)}d
\norm{\nabla\chi_E}}.
\end{eqnarray*}
Since $\nabla u_k \to \nabla \chi_E$ weak* and $\FF$ is continuous,
and noting that $r_j$ can be chosen such that
$\norm{\nabla\chi_E}(\partial B(y,r_j))=0$, we obtain
\begin{eqnarray*}
(\mathscr{F}_i \cdot \nnu)(y) &=& \lim_{j\to \infty}
\frac{\int_{B(y,r_j)}\FF \cdot \nabla \chi_E}{\int_{B(y,r_j)}d
\norm{\nabla\chi_E}}=\lim_{j\to\infty} \frac{\int_{B(y,r_j)}\FF(x)
\cdot \nnu(x)\,\, d\norm{\nabla
\chi_E}(x)}{\int_{B(y,r_j)}d\norm{\nabla\chi_E}(x)}\\
 &=&\FF(y)\cdot
\nnu(y),
\end{eqnarray*}
by differentiation of measures.
\end{proof}

The following corollary gives more information of the trace
$\sigma_i$ and the level sets $u_k^{-1}(s)$ when $s\to \frac{1}{2}+
$.

\begin{corollary}
The trace measure $ \sigma_i $ given in Theorem {\rm \ref{main}}
satisfies
\begin{eqnarray*}
\sigma_i(\rn)&=&2\lim_{k\to\infty}\int_E \FF \cdot \nabla u_k\, dy
=2\lim_{k \to \infty} \lim_{s \to \frac{1}{2}+} \int_{A_{k;s}} \FF
\cdot \nabla u_k\, dy \\
&=& 2 \lim_{k \to \infty} \int_{A_{k;\frac{1}{2}}} \FF \cdot \nabla
u_k\, dy.
\end{eqnarray*}
\end{corollary}

\begin{proof}
Theorem \ref{main} {\bf(h)} shows
\begin{equation}
\sigma_i(\rn) = 2 \lim_{k \to \infty} \int_E \FF \cdot \nabla u_k\,
dy.
\end{equation}
Using Lemma \ref{connection}, we find
\begin{eqnarray*}
\sigma_i(\rn)&=&2 \lim_{s \to \frac{1}{2}+} \int_s^1 \sigma_i(\rn)\,
dt = 2 \lim_{s \to \frac{1}{2}+} \lim_{k \to \infty}\int_s^1
\int_{u_k^{-1}(t)} \FF\cdot\nnu\, d\H^{N-1}dt\\
&=&2 \lim_{s \to \frac{1}{2}+} \lim_{k \to \infty} \int_{A_{k;s}} \FF
\cdot \nabla u_k\, dy.
\end{eqnarray*}
One can easily verify that the limits $s\to\frac {1}{2}+$ and $k \to
\infty$ can be interchanged. Noting that
$\medcup_{s>\frac{1}{2}} A_{k;s} = A_{k;\frac{1}{2}} , $
we conclude
\begin{eqnarray*}
\sigma_i(\rn)=2\lim_{k \to \infty} \lim_{s \to \frac{1}{2}+}
\int_{A_{k;s}} \FF \cdot \nabla u_k\, dy =2\lim_{k \to
\infty}\int_{A_{k;\frac{1}{2}}} \FF \cdot\nabla u_k\, dy.
\end{eqnarray*}
\end{proof}

\section{One-sided  approximation of sets of finite perimeter}
It is well-known that, a set of finite perimeter, \(E\), cannot be
approximated by smooth sets that lie completely in the interior of
\(E\).  For example, consider the open unit disk with a single
radius removed, and let $U$ be the resulting open set. Then the
Hausdorff measure of the boundary of \(U\) is $2 \pi$ plus the
measure of the radius, while the Hausdorff measure of the reduced
boundary is $2 \pi$. Thus, if \(U_{k}\) is an approximating open
subset of $U$, then its boundary will be close to that of boundary
$U$ and so its the Hausdorff measure will be close to $2 \pi$ plus
$1$. Adding more radii, say $m$ of them, will force the
approximating set to have boundaries whose Hausdorff measure close
to $2 \pi$ plus $m$. In general, if we let $K$ denote any compact
subset without interior and of infinite Hausdorff measure, then the
approximating sets will have boundaries whose measures will
necessarily tend to infinity.  On the other hand, the one-sided
approximation is possible for open sets of class \(C^{1}\), see
Theorem \ref{GGforC1}. More generally we have the following:
\begin{proposition}\label{31.1a}
Let \(U\subset \rn \) be an open set with \(\H^{N-1} (\partial U)<
\infty \). Then there exists a sequence of
bounded open
sets \(U_{k}\subset \overline{U_{k} }\subset U\) such that
\begin{enumerate}[\rm (i)]
\item $|U_{k}|= |\overline{U_{k}}|$;
\item \(\abs{U_{k} }\to \abs{U} \);
\item\(\H^{N-1}
(\partial U_{k} )\to \H^{N-1} (\partial U)\).
\end{enumerate}
\end{proposition}

\begin{proof}
By definition, for each integer \(k\), there exists a covering of
\(\partial U\) by balls
\[
\partial U\subset \medcup
  B_{i}(r_{i}),
\]
each with radius \(r_{i}\), such that
\[
\sum _{i=1}^{\infty }\H^{N-1}\big( \partial B_{i}(r_{i})\big )=\sum
_{i=1}^{\infty }\omega_{N-1}r_{i}^{N-1}<\H^{N-1} (\partial
U)+\frac{1}{k},
\]
where \(\omega_{N-1}\) is the \(\H^{N-1} \) measure of the boundary
of the unit ball in \(\rn \). Since \(\partial \Omega \) is compact,
the covering may be taken as a finite covering, say by \(m\) of
them, $B_{1}(r_{1}),B_{1}(r_{2}), \dots$, $B_{m}(r_{m})$. Then the
open set \(V_{k} :=  \medcup B_{i}(r_{i}) \) has the property that
\[
\partial V_{k}\subset\medcup_{i=1}^{m} \partial B_{i}(r_{i})
\]
and therefore that
\[
\H^{N-1} (\partial V_{k})\le \H^{N-1} (\medcup_{i=1}^{m} \partial
B_{i}(r_{i})) \le\sum _{i=1}^{\infty
}\omega_{N-1}r_{i}^{N-1}<\H^{N-1} (\partial U)+\frac{1}{k}.
\]
Thus, the open sets \(U_{k} := U \setminus \overline V_{k} \subset
U\) will satisfy our desired result, except that they are not
smooth.
\end{proof}

Given an arbitrary set of finite perimeter, \(E\), we know from \S 4
that \(E\) can be approximated by sets with smooth boundaries
essentially from the measure-theoretic interior of $E$, that is, a
one-sided approximation can ``almost'' be achieved (see Theorem
\ref{main}{\bf (e)}). On the other hand, the next result shows that,
if \(E\) is sufficiently regular, there does, in fact, exist a
one-sided approximation. The condition of regularity we impose is
similar to Lewis's {\it uniformly flat} condition in potential
theory \cite{Lewis}.

\begin{theorem} \label{estesi}
Suppose that $E$ is a bounded set of finite perimeter with the
property that, for all \(y\in \partial E\), there  are positive
constants \(c_{0}\) and \(r_{0}\) such that
\begin{equation}\label{8.1}
\frac{\abs{E^{0}\cap B(y,r)}}{|B(y,r)|}\ge c_{0} \:\qquad \text{for
all \(r\le r_{0}\)}.
\end{equation}
Then there exists $t \in (0,1)$ such that
\begin{equation}\label{8.2}
 A_{k;t} \Subset E \qquad \text{ for large \(k\)}.
\end{equation}
\end{theorem}

\begin{proof}
Choose a mollifying kernel $\rho$ such that $\rho=1$ on
$B(0,\frac{1}{2})$. If $y \in \po E$, we have
\begin{align*}
 v_k(y):= \charfn{\rn \setminus E}* \rho_{\ve_k} (y)
 & =\frac{1}{\ve_k^N}\int_{B(y,\ve_k)} \charfn{\rn \setminus E}(x)
 \rho(\frac{x-y}{\ve_k})\, dx \\
    & \ge\frac{1}{\ve_k^N} \int_{B(y,\frac{\ve_k}{2})}
  \charfn{\rn \setminus E}(x)\, dx\\
&  =\frac {\vert (\rn \setminus E) \cap B(y, \frac{\ve_k}{2})
\vert}{\ve_k^N}\\
&=\frac {\vert E^{0} \cap B(y, \frac{\ve_k}{2}) \vert}{\ve_k^N} \ge
c_{0}/2^{N}:= \tilde c_{0},
\end{align*}
where $0<\tilde{c_0}<1$ depends only on the dimension $N$ and is
independent of the point $y$. Note that \(u_{k}(y)+v_{k}(y)=1\) for
all \(y\in \rn \).
Therefore, for all $y\in \po E$,
$$
u_k(y)=1-v_k(y) \leq 1- \tilde{c_0}.
$$
Thus, taking \(1-\tilde c_{0}<t<1\), we see that \(A_{k;t}\cap
\partial E=\emptyset\). Consequently, each connected component
 of the open set \(A_{k;t}\) lies either in the interior of \(E\) or
in
its exterior, and thus must lie in its interior.
\end{proof}

\begin{corollary} \label{yacasiesta}
Let $E$ be a bounded set of finite perimeter with uniform Lipschitz
boundary. Then there exists $T \in (0,1)$ such that $A_{k;T} \Subset
E$.
\end{corollary}

\begin{proof}
Since $E$ has a uniform Lipschitz boundary, for each \(x\in
\partial E \), there is a finite cone, \(C_{x} \), with vertex \(x
\) that completely lies in the complement of \(E\). Each cone
\(C_{x} \) is assumed to be congruent to a fixed cone \(C \). This
implies that the hypothesis of Theorem \ref{estesi} is satisfied.
Therefore, there exists \(0<T<1 \) such that \(u_{k}(y)<T \) for all
$k$ and all $y\in\partial E$.
\end{proof}

\begin{definition}
An open set \(U \subset \rn \) is called an {\bf extension domain}
  for \(\FF \in \mathcal{DM}^{\infty }(U ) \),  if there exists a field \(\FF
^{*}\in \mathcal{DM}^{\infty
  }(\rn ) \) such that \(\FF =\FF ^{*}\) on \(U \).
 \end{definition}

\begin{theorem} An open set
\(U \) satisfying $\mathcal{H}^{N-1}(\partial U)<\infty$ is an
extension domain for any $\FF \in \mathcal{DM}^{\infty }(U )$. More
generally, if $\FF\in \DM^\infty_{loc}(\Omega)$, then any open set
of finite perimeter $U\Subset\Omega$  is an extension domain for
$\FF$.
\end{theorem}

\begin{proof} We define an extension of \(\FF \) by
\[
\FF^{*}(y):= \charfn{U }(y) \FF(y) \quad\text{for all \(y\in \rn \).}
\]
According to Definition \ref{dmfield}, it suffices to
  show that
\[
 \sup\left\{\int_{\rn }\FF^{*} \cdot \nabla \varphi:\abs{ \varphi
    }\le1,  \varphi \in
    C^{\infty}_{c}(\rn  ) \right \}<\infty.
\]

We consider first the case $\mathcal{H}^{N-1}(U)<\infty$. Let $U_k$
be the sequence of approximate sets given in Proposition
\ref{31.1a}. Therefore, for any $\varphi\in C_c^\infty(\rn)$ with
$|\varphi|\le 1$, we employ our general Gauss-Green theorem, Theorem
\ref{main}, to obtain
\[
\int_{U_k} \FF \cdot \nabla \varphi\, dy+\int_{U_k}\varphi\, \div
{\FF }=-\int_{\partial U_k}\varphi\, \mathscr{F}_i\cdot \nnu \di
\H^{N-1}.
\]
Thus,
\begin{align*}
\int_{U_k} \FF\cdot \nabla \varphi\, dy &=-\int_{U_k}\varphi\, \div
{\FF } -\int_{\partial U_k}\varphi \,  \mathscr{F}_i\cdot \nnu
\di \H^{N-1}\\
&\le \|\div {\FF}\|(U_k)+\norm{\FF
}_{\infty}\mathcal{H}^{N-1}(\partial U_k)\\
&\le \|\div {\FF}\|(U)+\norm{\FF
}_{\infty}\mathcal{H}^{N-1}(\partial U_k).
\end{align*}
Letting $k\to\infty$, we obtain
$$
\int_{U} \FF\cdot \nabla \varphi\, dy \le \|\div
{\FF}\|(U)+\norm{\FF }_{\infty}\mathcal{H}^{N-1}(\partial U)<\infty.
$$
Thus,
$$
\int_{U} \FF^*\cdot \nabla \varphi\, dy=\int_{U} \FF\cdot \nabla
\varphi\, dy<\infty.
$$

We now consider the case that $U\Subset \Omega$ is a set of finite
perimeter and $\FF\in \DM^\infty_{loc}(\Omega)$. Proceeding as above
and using Theorem \ref{main},
\begin{align*}
\int_{\rn} \FF^* \cdot \nabla \varphi\, dy &= \int_{U} \FF^*\cdot
\nabla \varphi\, dy \\
&=-\int_{U}\varphi\, \div {\FF } -\int_{\partial^* U}\varphi \,
\mathscr{F}_i\cdot \nnu
\di \H^{N-1}\\
&\le \|\div {\FF}\|(U)+\norm{\FF
}_{\infty}\mathcal{H}^{N-1}(\partial^* U)<\infty.
\end{align*}
\end{proof}

\begin{cor} [Chen-Torres \cite{ChenTorres}]
Let \(U\subset \rn \) be
a  bounded, open set with \(\H^{N-1}(\partial U)<\infty\). Let
  \(\FF_{1}\in \mathcal{DM}^{\infty } (U)\) and
\(\FF_{2}\in \mathcal{DM}^{\infty}(\rn\setminus \overline {U}).\)
Then, with
\[
\FF(y):=  \begin{cases}
\FF_{1}(y)&\quad \text{$y\in U$},\\
\FF_{2}(y)&\quad \text{$y\in \rn \setminus \overline{U}$,}\\
\end{cases}
\]
we have
\[
\FF\in \mathcal{DM}^{\infty}(\rn).
\]
\end{cor}

\begin{proof} Applying the previous result to
\[
\FF^{*}_{1}:= \charfn{U}(y) \FF_{1}(y) \quad\text{for all \(y\in \rn
\)}
\]
and
\[
\FF^{*}_{2}:= \charfn{\rn \setminus \overline{U}} \FF_{2}(y)
\quad\text{for all \(y\in \rn \)},
\]
we see that
\[
\FF =\FF^{*} _{1}+\FF^{*} _{2}\qedhere
\]
\end{proof}

\medskip
\section{Cauchy Fluxes and Divergence-Measure
Fields}\label{sec:cauchy}

The physical principle of balance law of the form
\begin{equation}
\int_{\partial E} f(y, \nnu(y))\, d\H^{N-1}(y) +\int_{E} b(y)\, dy =0
\label{balance}
\end{equation}
is basic in all of classical physics. Here, $\nnu(y)$ is the
interior unit normal to the boundary $\partial E$ of $E$. In
mechanics, $f$ represents the surface force per unit area on
$\partial E$, while in thermodynamics $f$ gives the heat flow per
unit area across the boundary $\partial E$.

\medskip
In 1823, Cauchy \cite{Cauchy1} (also see \cite{Cauchy2}) established
the {\it stress theorem} that is probably the most important result
in continuum mechanics: If both $f(y,\nnu(y))$, defined for each $y$
in an open region $\Omega$ and every unit vector $\nnu$, is
continuous in $y$ and $b(y)$ is uniformly bounded on $\Omega$, and
if (\ref{balance}) is satisfied for every smooth region
$E\Subset\Omega$, then $f(y,\nnu)$ must be linear in $\nnu$. The
{\it Cauchy postulate} states that the density flux $f$ through a
surface depends on the surface solely through the normal at that
point. For instance, if $f(y,\nnu)$ represents the heat flow, then
the Stress theorem states that there exists a vector field $\FF$
such that
$$
f(y,\nnu)=\FF(y) \cdot \nnu.
$$

\medskip
Since the time of Cauchy's stress result, \cite{Cauchy1,Cauchy2},
many efforts have been made to generalize his ideas and remove some
of his hypotheses. The first results in this direction were obtained
by Noll \cite{Noll} in 1959, who set up a basis for an axiomatic
foundation for continuum thermodynamics. In particular, Noll
\cite{Noll} showed that the Cauchy postulate may directly follow
from the balance law. In \cite{gm}, Gurtin-Martins introduced the
concept of Cauchy flux and removed the continuity assumption on $f$.
In \cite{Ziemer1}, Ziemer proved Noll's theorem in the context of
geometric measure theory, in which the Cauchy fluxes were first
formulated at the level of generality with sets of finite perimeter
in the absence of jump surfaces, ``shock waves".

However, as we explain below, all the previous formulations of
(\ref{balance}) do not allow the presence of ``shock waves''; one of
our main intentions in this paper is to develop a theory that will
allow the presence of ``shock waves''.

In this section we first introduce a class of Cauchy fluxes that
allows the presence of the exceptional surfaces or ``shock waves'' and
then prove that such a Cauchy flux induces a bounded
divergence-measure (vector) field $\FF$ so that the Cauchy flux over
{\it every} oriented surface
can be recovered
through $\FF$ and the normal to the oriented surface.
Before introducing this framework, we need the following
definitions.

\begin{definition} An oriented surface in $\Omega$ is a pair $(S,\nnu)$
so that $S\Subset\Omega$ is a Borel set and $\nnu:\R^N \rightarrow
\SB^{N-1}$ is a Borel measurable unit vector field that satisfy the
following property: There is a set $E\Subset\Omega$ of finite
perimeter such that $S \subset \po^*E$ and
\begin{equation}
\nnu(y)=\nnu_{_{E}}(y)\, \chi_{S}(y), \nonumber
\end{equation}
where $\chi_{S}$ is the characteristic function of the set $S$ and
$\nnu_{_{E}}(y)$ is the interior measure-theoretic unit normal to $E$
at $y$.
\end{definition}

Two oriented surfaces $(S_j,\nnu_j), j=1,2$, are said to be
compatible if there exists a set of finite perimeter $E$ such that
$S_j\subset \po^{*}E$ and $\nnu_j(y)=\nnu_{_{E}}(y)\, \chi_{S_j}(y)$,
$j=1,2.$ For simplicity, we will denote the pair $(S,\nnu)$ simply as
$S$, with implicit understanding that $S$ is oriented by the
interior normal of some set $E$ of finite perimeter. We define $-S=
(S,-\nnu)$, which is regarded as a different oriented surface.

\begin{definition}
\label{def4} Let $\Omega$ be a bounded open set. A Cauchy flux is a
functional $\F$ that assigns to each oriented surface
$S:=(S,\nnu)\Subset \Omega$ a real number and has the following
properties:
\begin{enumerate} [\rm (i)]
\item $\F(S_1 \cup S_2) = \F(S_1)+ \F(S_2)$ for any pair of compatible
      disjoint surfaces $S_1, S_2 \Subset\Omega$; \label{uno}

\item There exists a nonnegative Radon measure $\sigma$ in $\Omega$
such that
       $$
      |\F(\po^* E)| \leq \sigma (E)
      $$ \label{tres}
for every set of finite perimeter $E \Subset\Omega$ satisfying
$\sigma (\po E)=0$;

\item There exists a constant $C$ such that
      $$
      |\F(S)| \leq C\, \mathcal{H}^{N-1}(S)
      $$
for  every oriented surface $S \Subset\Omega$ satisfying
$\sigma(S)=0$.
   \label{dos}
\end{enumerate}
\end{definition}

This general framework for Cauchy fluxes allows the presence of
exceptional surfaces, ``shock waves'', in the formulation of the
axioms, on which the measure $\sigma$ has support. On these
exceptional surfaces, the Cauchy flux $\F$ has a discontinuity,
i.e., $\F(S)\ne -\F(-S)$. In fact, the exceptional surfaces are
supported on the singular part of measure $\sigma$. When $\sigma$
reduces to the $N$-dimensional Lebesgue measure $\mathcal{L}^N$, the
formulation reduces to Ziemer's formulation in \cite{Ziemer1} and in
this case $\sigma$ vanishes on any $\H^{N-1}$-dimensional surface,
which excludes shock waves.

The theory developed in this paper allows to approximate the
exceptional surfaces or ``shock waves'' with smooth boundaries and
rigorously pass to the limit to recover the flux across the shock
waves. This allows to capture measure production density in the
formulation of the balance law and entropy dissipation for entropy
solutions of hyperbolic conservation laws. Once we know the flux
across  {\it every surface}, we proceed to obtain a rigorous
derivation of nonlinear systems of balance laws with measure source
terms from the physical principle of balance law in \S 10. The
framework also allows the recovery of Cauchy entropy fluxes through
the Lax entropy inequality for entropy solutions of hyperbolic
conservation laws by capturing entropy dissipation; see \S 11.

The main theorem of this section is the following.

\begin{theorem}\label{thm:9.1}
Let $\F$ be a  Cauchy flux in $\Omega$. Then there exists a unique
divergence-measure field $\FF \in
\mathcal{DM}^{\infty}_{loc}(\Omega)$ such that
\begin{equation}
\F (S) = -\int_S \mathscr{F}_i\cdot \nnu\, d\H^{N-1} \label{ayuda:a}
\end{equation}
for {\bf every} oriented surface $(S,\nnu)\Subset \Omega$, where
$\mathscr{F}_i\cdot\nnu$ is the normal trace of $\FF$ to the
oriented surface.
\end{theorem}

When $\sigma$ reduces to the $N$-dimensional Lebesgue measure
$\mathcal{L}^N$, as in Ziemer's formulation, the vector field $\FF$
satisfies $\div \FF \in L^\infty$ and $\F(S)=-\F(-S)$ for every
surface $S$, which thus excludes shock waves where the Cauchy flux
$\F$ has a discontinuity, i.e., $\F(S)\ne -\F(-S)$.

In order to establish Theorem \ref{thm:9.1}, we need  Lemmas
\ref{t3}--\ref{t1} that were first shown in
Degiovanni-Marzocci-Musesti \cite{Italians}. Here we offer
simplified proofs of these facts for completeness. In particular,
Lemma \ref{t3} is in fact a direct application of Theorem \ref{ta1}
(due to Fuglede) below, and Lemma \ref{t1} follows by an
approximation and Theorem \ref{main}. We also refer to Schuricht
\cite{German} for a different approach in formulating the axioms in
Definition \ref{def4}.

\medskip
The following theorem, due to Fuglede, is a generalization of
Riesz's theorem, whose proof can be found in \cite{Fuglede1}.

\begin{theorem}\label{ta1} Let $\mu$ be a nonnegative measure defined on a
$\sigma$-field $\mathcal{V}$ of subsets of a fixed set $X$ and $X
\in \mathcal{V}$. Let $\varphi$ be an additive set function defined
on a system of sets $\mathcal{U} \subset \mathcal{V}$ such that all
finite unions of disjoint sets from $\mathcal{U}$, together with the
empty set, form a field $\mathcal{F}$ which generates $\mathcal{V}$.
Assume that $\mu(A) < \infty$ for every $A\subset\mathcal{U}$. Then
there exists a function $g(y) \in L^1(X,\mathcal{V},\mu)$ with the
property that
$$
\varphi(A)=\int_A g(y)\, d\mu \quad \textnormal{ for every } A \in
\mathcal{U}
$$
if and only if the following hold:
\begin{enumerate}[\rm (i)]
\item For every $\ve>0$, there exists $\delta >0$ such that
$\sum_{i=1}^n |\varphi(A_i)|\leq\ve$ for every finite system of
disjoint sets $A_1, A_2,..., A_n$ from $\mathcal{U}$ for which
$\sum_{i=1}^n \mu(A_i) < \delta$; \label{primero}
\item There is a finite constant $C$ such that
 $\sum_{i=1}^n |\varphi(A_i)| \leq C$ for every finite system
of disjoint sets $A_1, A_2,..., A_n$ from $\mathcal{U}$.
 \label{segundo}
\end{enumerate}
The function $g$ is then essentially uniquely determined. Under the
additional assumption that $\mu(X)<\infty$, condition {\rm (ii)} is
a consequence of condition {\rm (i)}.
\end{theorem}

Let $\{I\}$ be the collection of all closed cubes in $\R^{N}$ of the
form
$$
I=[a_1,b_1]\times [a_2,b_2] \times ... [a_N,b_N],
$$
where $a_1,b_1,a_2,b_2,...,a_N$, and $b_N$ are real numbers. For
almost every $\tau_j\in [a_j, b_j]$, we define
$$
I_{\tau_j} = \{y \in I: y_j=\tau_j \}.
$$
We define the vectors $e_1,e_2,...,e_N$ so that the $j$-th component
of $e_j$ is $-1$ and the other components are zero. We orient the
surface $I_{\tau_j}$ with the vector $e_j$.

\begin{lemma}\label{t3}
Let $\F$ be a Cauchy flux in $\Omega$. Then there exists a
divergence-measure field $\FF \in
\mathcal{DM}^{\infty}_{loc}(\Omega)$ such that, for every  cube
$I=[a_{1},b_{1}] \times [a_{2},b_{2}] \times ... \times
[a_{N},b_{N}]\Subset\Omega$ and almost every $\tau_j\in [a_{j},
b_{j}]$,
$$
\F(I_{\tau_j})=-\int_{I_{\tau_j}} \FF(y)\cdot e_j \, d
\mathcal{H}^{N-1}(y).
$$
\end{lemma}

\begin{proof}
{\it Step 1.} We fix $j \in \{1,...,N \}$. For every cube $I \subset
\Omega$, we define
$$
\mu^j (I)=\int_{a_{j}}^{b_{j}} \F(I_{\tau_j})\, d \tau_j.
$$
We have
\begin{equation}
\label{reason} |\mu^j(I)| \leq  \int_{a_{j}}^{b_{j}}
|\F(I_{\tau_j})|\, d\tau_j \leq C \int_{a_{j}}^{b_{j}}
\int_{I_{\tau_j}}  d\H^{N-1} d\tau_j = C |I|.
\end{equation}
Thus, from Theorem \ref{ta1}, there exists a function $f^j \in
L^{1}(\Omega)$ such that
$$
 \mu^j(I) = \int_I f^j\, dy \qquad \tn{  for every }I.
$$
In fact, inequality (\ref{reason}) implies that $f^j \in
L^{\infty}(\Omega)$ since, for $\LL^N$--a.e. $y$,
$$
 f^j(y)=\lim_{|I| \to 0, y \in I} \frac{\int_I f^j \, dx}{|I|} \leq C.
$$

Fubini's theorem implies that
\begin{equation}
\label{que}
 \mu^j(I)=\int_{a_{j}}^{b_{j}} \F(I_{\tau_j})\, d\tau_j
= \int_{I} f^{j}\, dy= \int_{a_{j}}^{b_{j}} \int_{I_{\tau_j}} f^j
d\H^{N-1} d\tau_j .
\end{equation}
Let $\tau_j\in [a_{j}, b_{j}]$, $\alpha_{k,j}$, and $\beta_{k,j}$ be
sequences such that
$$
\alpha_{k,j} \leq \tau_{j} \leq \beta_{k,j},
$$
where $\alpha_{k,j}$ is an increasing sequence that converges to
$\tau_{j}$ as
$k \to \infty$, and $\beta_{k,j}$ is a decreasing  sequence that
converges to $\tau_{j}$ as $k \to \infty$. Thus, from (\ref{que}),
we obtain
\begin{equation}
\frac{1}{\alpha_{k,j}-\beta_{k,j}}\int_{a_{j}}^{b_{j}}\F(I_{\tau_j})\,
d\tau_j = \frac{1}{\alpha_{k,j} -\beta_{k,j}}\int_{a_{j}}^{b_{j}}
\int_{I_{\tau_j}} f^{j}\, d\H^{N-1}d\tau_j.
\end{equation}
We let $k \to \infty$ to obtain that, for a.e. $\tau_j$,
\begin{equation}
\F(I_{\tau_{j}}) = \int_{I_{\tau_{j}}} f^{j}\, d\H^{N-1}.
\end{equation}
Define
$$
\FF:=(f^1,f^2,...,f^N).
$$
Then we find that, for almost every $\tau_j$, $j \in \{1,2,...,N\}$,
\begin{equation}
\F(I_{\tau_{j}}) = -\int_{I_{\tau_{j}}} \FF(y) \cdot e_j\,
d\H^{N-1}(y).
\end{equation}

{\it Step 2.} We now prove that the divergence of $\FF$, in the
sense of distributions, is a measure. We define, for a.e. cube $I$,
\begin{equation}
\eta(I):=-\int_{\po I}\FF(y)\cdot \nnu(y)\, d\H^{N-1}(y),
\label{lamedida}
\end{equation}
where $\nnu$ is the interior unit normal to $\po I$.

{}From Step 1 and the definition of Cauchy flux, we have
\begin{equation}
\abs{\eta(I)}=\abs{\F(\partial I)} \leq \sigma (I)  \label{matadas}
\end{equation}
for almost all closed cubes. Thus, we can again apply Theorem
\ref{ta1} to conclude that there exists a function $g \in
L^1(\Omega; \sigma)$, uniquely defined in $\Omega$ up to a set of
$\sigma$-measure zero, such that
\begin{equation}
\eta(I)=-\int_{\po I}\FF(y)\cdot \nnu(y)\,
d\H^{N-1}(y)=\int_{I}g(y)\,d \sigma \label{matadas2}
\end{equation}
for almost every closed cube $I \subset\Omega$.

Denote $\tilde{\sigma}$ the measure given by $gd\sigma$ in $\Omega$.
We now prove
\begin{equation}
\textnormal{div $\FF = \tilde{\sigma}$ }
\end{equation}
in the sense of distributions in any open set $U\Subset\Omega$.

Let $I \Subset U$ be any closed cube. Then, for any $\phi \in C^1$
with support contained in $I$,
\begin{equation}
\int_{U} \FF \cdot \nabla \phi\, dy =\lim_{\ve\to 0}\int_{U}
\FF_\ve\cdot\nabla\phi\, dy =-\lim_{\ve\to
0}\int_{U}\phi\,\textnormal{div} \FF_\ve\, dy, \label{casi}
\end{equation}
where $\FF_{\ve}=\FF*\rho_{\ve}$ and $\rho$ is the standard
mollifying kernel. We now prove that, for $\LL^N$-a.e. $y \in U$,
$$
\tilde{\sigma}_\ve(y)=\textnormal{div} \FF_\ve (y),
$$
where $\tilde{\sigma}_{\ve}$ is the convolution of function
$\rho_{\ve}$ with the measure $\tilde{\sigma}$; that is,
\begin{equation}
\tilde{\sigma}_{\ve}(y):=(\rho_{\ve}*\tilde{\sigma})(y)
=\int_{\Omega} \rho_{\ve}(y-x)\, d\tilde{\sigma}(x) \label{fcm}.
\end{equation}

{}From (\ref{matadas})--(\ref{matadas2}), we find that, for $\ve <
dist (\po U, \po\Omega)$,
\begin{eqnarray*}
\int_{I}\div \FF_\ve(y)\,dy &= & -\int_{\po I}\FF_\ve(y)\cdot
\nnu(y)\, dy
=  -\int_{\po I}\int_{\R^{N}}\FF(y-x)\cdot\nnu(y)\rho_\ve(x)\, dxdy\\
&= & -\int_{\R^{N}}\int_{\po I}\FF(y-x)\cdot \nnu(y) \rho_\ve(x)\, dydx\\
&= & -\int_{\R^{N}} \Big(\int_{\po I_x}\FF(y)\cdot \nnu(y)dy \Big)
\rho_\ve(x)\, dx\\
&= & \int_{\R^{N}} \tilde{\sigma}(I_x) \rho_\ve(x)\, dx,
\end{eqnarray*}
where $I_x=\{y\, :\, a_i\le y_i-x_i\le b_i, \, i=1, \cdots, N\}$. We
can consider the smooth function $\rho_{\ve}$ as a measure in
$\R^{N}$, say $\lambda_{\ve}$, by defining $\lambda_{\ve}(A)
=\int_{A} \rho_{\ve}(x)\, dx$ for any Borel set $A$. We can also
extend the measure $\tilde{\sigma}$ by zero outside $\Omega$.
Therefore, we find
\begin{equation}
\int_{\R^{N}} \tilde{\sigma}(I_x) \rho_\ve(x)\, dx =
(\tilde{\sigma}* \lambda_{\ve})(I) = (\lambda_{\ve} *
\tilde{\sigma})(I) =\int_{\R^{N}} \lambda_{\ve}(I_x)\,
d\tilde{\sigma}(x). \label{riguroso}
\end{equation}
{}From (\ref{riguroso}) and using (\ref{fcm}), we compute
\begin{eqnarray*}
\int_{\R^{N}} \lambda_{\ve}(I_x)\, d\tilde{\sigma}(x)
=\int_{\Omega}\lambda_{\ve}(I_x)\, d\tilde{\sigma}(x)
&=&\int_{\Omega}\Big(\int_{I_x}\rho_{\ve}(y)dy\Big)d\tilde{\sigma}(x)
=\int_{\Omega}\int_{I}\rho_{\ve}(y-x)\, dy d\tilde{\sigma}(x) \\
&=& \int_{I}\Big(\int_{\Omega} \rho_{\ve}(y-x)\, d\tilde{\sigma}(x)
\Big) dy
= \int_{I}(\rho_{\ve}*\tilde{\sigma})(y)\, dy \\
&= & \int_{I} \tilde{\sigma}_{\ve}(y)\, dy.
\end{eqnarray*}
Therefore,
$$
\int_{I}\textnormal{\div}\FF_\ve(y)\,dy =\int_{I}
\tilde{\sigma}_{\ve}(y)\, dy.
$$
Since the cube $I\Subset U$ is arbitrary, this shows that
$\tilde{\sigma}_\ve (y)= \tn{\div }\FF_\ve (y)$ for $\LL^N$-a.e. $y \in
U$. Using this in (\ref{casi}), we obtain
\begin{equation}
\int_{U}\FF\cdot \nabla \phi\, dy =-\lim_{\ve\to 0}
\int_{U}\phi\,\textnormal{\div}\FF_\ve\, dy =-\lim_{\ve \to 0}
\int_{U}\phi\,\tilde{\sigma}_\ve\, dy = -\int_{U}\phi(y)\, d
\tilde{\sigma}(y),
\end{equation}
since the sequence of measures $\tilde{\sigma}_\ve$ converges
locally weak* in $\Omega$ to $\tilde{\sigma}$ as $\ve \to 0$.
\end{proof}

\begin{lemma}
\label{t1} Let $\F$ be a  Cauchy flux in $\Omega$. Then there exists
a unique divergence-measure field $\FF \in
\mathcal{DM}^{\infty}_{loc}(\Omega)$ such that
\begin{equation}
\F (S) = -\int_S \mathscr{F}_i \cdot \nnu\, d\H^{N-1} \label{ayuda}
\end{equation}
for almost every oriented surface $(S,\nnu)\Subset\Omega$; that is,
every surface $S$ in $\Omega$ satisfying the condition that $\sigma
(S)=0$.
\end{lemma}

\begin{proof}
Using Lemma \ref{t3}, it follows that there exists an
$\FF\in\DM^\infty_{loc}(\Omega)$ such that, for any cube
$I=[a_1,b_1]\times [a_2,b_2] \times ... \times
[a_N,b_N]\Subset\Omega$,
\begin{equation}
\label{base}
 \F (I_{\tau_j}) = -\int_{I_{\tau_j}} \FF(y) \cdot e_j\,
d\H^{N-1}(y)
\end{equation}
for almost every $\tau_j\in [a_j,b_j]$.

Let $(S,\nnu)$ be an oriented surface satisfying $\sigma(S)=0$.
Then, since $\div \FF=g \, d\sigma$ and proceeding as in Lemma
\ref{consistency}, we conclude that $\|\div \FF\|(S)=0$. We
approximate $S$ with closed cubes such that
\begin{equation}
  S = \bigcap_{i=1}^{\infty} J_{i}, \label{inter}
\end{equation}
where each $J_i$ is a finite union of closed cubes and $J_{i+1}
\subset J_i$. Since $S$ is an oriented surface, there exists a set
of finite perimeter, $E$, such that $S \subset \po^*E$. Using Lemma
\ref{tedious}, we have
\begin{equation}\label{descomponer}
\po^* (J_i \cap E) = S \cup (\po^* J_i \cap E) \cup \mathcal{N}_i,
\end{equation}
where $\lim\limits_{i \to \infty} \H^{N-1}(\mathcal{N}_i)=0$ and
thus, since
 $\F(\mathcal{N}_i)\leq C\, \H^{N-1}(\mathcal{N}_i)$, we obtain
\begin{equation}\label{9.20}
    \lim_{i \to \infty} \F(\mathcal{N}_i)=0.
\end{equation}

The definition of Cauchy flux implies that
\begin{equation*}
|\F(\po^* (J_i \cap E))| \leq \sigma (J_{i} \cap E),
\end{equation*}
and thus (\ref{inter}) implies that
\begin{equation}
\label{antes11}
 \lim_{i \to \infty} |\F(\po^* (J_i \cap E))| \leq
\sigma(S)=0.
\end{equation}
On the other hand, using Theorem \ref{main}, we have
\begin{equation}
\Big|-\int_{\po^* (J_i \cap E)} \mathscr{F}_i \cdot \nnu\,
d\H^{N-1}\Big|=\Big|\int_{J_i \cap E}\div \FF\Big|\leq \norm{\div
\FF}(J_i \cap E),
\end{equation}
which yields
\begin{equation}
\label{despues1}
 \lim_{i \to \infty}\Big|-\int_{\po^* (J_i \cap E)}\mathscr{F}_i
\cdot \nnu\, d\H^{N-1}\Big| \leq \norm{\div \FF}(S)=0.
\end{equation}
Using (\ref{descomponer})--\eqref{antes11} and Lemma \ref{t3},
we obtain
\begin{eqnarray}
\label{antes2}
 \lim_{i \to \infty} |\F(\po^* (J_i \cap E))|
 &=&\lim_{i \to \infty} |\F(S)+\F(\po^* J_i \cap E)|\\
 &=&\lim_{i \to \infty} \Big|\F(S)-\int_{\po^* J_i \cap E}\FF(y) \cdot
 \nnu(y)\, d\H^{N-1}\Big|=0. \nonumber
\end{eqnarray}
{}From (\ref{descomponer}), \eqref{9.20}, and (\ref{despues1}), we
obtain
\begin{equation}
\label{despues2}
 \lim_{i \to \infty}\Big|-\int_{\po^* J_i \cap E}\FF(y)
\cdot \nnu(y)\, d\H^{N-1} - \int_{S}\mathscr{F}_i \cdot \nnu \,
d\H^{N-1}\Big| =0.
\end{equation}
Combining (\ref{antes2}) with (\ref{despues2}) yields
\begin{equation*}
    \F(S)=- \int_{S}\mathscr{F}_i \cdot \nnu\, d\H^{N-1}.
\end{equation*}

Assume now that there exists another vector field
$G=(g^1,g^2,...,g^N)$ such that (\ref{ayuda}) holds. Then, for fixed
$j \in \{1,2,...,N\}$, we have
\begin{equation}
\int_I f^j dy  =  \int_{a_{j}}^{b_{j}} \int_{I_{\tau_j}} f^{j}
d\H^{N-1} d\tau_j =  \int_{a_{j}}^{b_{j}} \int_{I_{\tau_j}} g^{j}
d\H^{N-1}\, d\tau_j =  \int_{I} g^j\, dy
\end{equation}
for any cube $I$. This implies that
$$
f^j(y)=g^j(y) \qquad \text{for almost every $y$}.
$$
\end{proof}

With Lemmas \ref{t3}--\ref{t1}, we now prove Theorem \ref{thm:9.1}
to explain how the Cauchy flux can be recovered on the exceptional
surfaces based on the theory established in \S 3--\S 8.

\begin{proof}[Proof of Theorem {\rm \ref{thm:9.1}}]
 Let $(S,\nnu)$
be an oriented surface on which $\sigma(S)\ne 0$, i.e. $\F(-S)\ne -
\F(S)$. By definition of oriented surfaces, there exists a bounded
set of finite perimeter, $E:=E^1\cup
\partial^*E\Subset\Omega$,  such that
\begin{equation}
S \subset \partial^* E\quad \text{ and }\quad \nnu(y)=\nnu_E(y)
\chi_S(y),
\end{equation}
where $\nnu_E(y)$ is the interior normal to $E$ at $y\in S$. Consider
$$
 \tilde{E} = E^0\cup \partial^*E.
$$
Then Theorem \ref{main} implies that there exist the normal traces
$\mathscr{F}_{i}\cdot \nnu$ and $\mathscr{F}_{e}\cdot\nnu$ defined on
$\po^*\tilde{E}=\po^* E$ respectively such that
\begin{eqnarray*}
&&\int_{E^{1}} \div\FF =-\int_{\partial^{*}E}
\mathscr{F}_{i}\cdot \nnu\, d\mathcal{H}^{N-1},\\
&&\int_{E} \div\FF =-\int_{\partial^*\tilde E}
\mathscr{F}_{e}\cdot \nnu\, d\mathcal{H}^{N-1} =-\int_{\partial^*E}
\mathscr{F}_{e}\cdot \nnu\, d\mathcal{H}^{N-1}.
\end{eqnarray*}

Theorem \ref{main} indicates that the traces $\mathscr{F}_{i}\cdot
\nnu$ and $\mathscr{F}_{e}\cdot \nnu$ can be recovered, up to a set of
arbitrary small $\H^{N-1}$-measure, from the neighborhood behavior
of the vector field $\FF$.
This observation allows us to define
\begin{equation}
 \F(S)=\F(S,\nnu):=-\int_S \mathscr{F}_{i}\cdot \nnu\, d\H^{N-1},
\end{equation}
and
\begin{equation}
 \F(-S)=\F(S,-\nnu):=-\int_S \mathscr{F}_{e}\cdot (-\nnu)\, d\H^{N-1}
 =\int_S \mathscr{F}_{e}\cdot \nnu\, d\H^{N-1}.
\end{equation}
In this way, we can recover the Cauchy flux $\F$ through the
corresponding divergence-measure field $\FF$ over all oriented
surfaces, especially including the exceptional surfaces. That is,
the normal traces of $\FF\in\DM^\infty_{loc}(\Omega)$ are the Cauchy
densities over all oriented surfaces. This completes the proof of
Theorem \ref{thm:9.1}.
\end{proof}

\medskip
\section{Mathematical formulation of the balance law
and Derivation of Systems of Balance Laws}\label{sec:balance}

In this section we first present the mathematical formulation of the
physical principle of  balance law (\ref{balance}). Then we apply
the results established in \S 3--\S 9 to give a rigorous derivation
of systems of balance laws with measure source terms. In particular,
we give a derivation of hyperbolic systems of conservation laws
(\ref{6.11}).

A balance law on an open subset $\Omega$ of $\R^N$ postulates that
the {\it production} of a vector-valued ``extensive'' quantity in
any bounded measurable subset $E\Subset \Omega$ with finite
perimeter is balanced by the {\it Cauchy flux} of this quantity
through the measure-theoretic boundary $\partial^{m}E$ of $E$ (see
Dafermos \cite{Da,Da-book}).

Like the Cauchy flux, the production is introduced through a
functional $\PP$, defined on any bounded measurable subset of finite
perimeter, $E\subset \Omega$, taking value in $\R^k$ and satisfying
the conditions:
\begin{eqnarray}
&&\PP(E_1\cup E_2)=\PP(E_1)+\PP(E_2)
  \qquad \text{if}\,\, E_1\cap E_2=\emptyset, \label{6.1} \\
&&|\PP(E)|\le \sigma(E). \label{6.2}
\end{eqnarray}

Then the physical principle of balance law can be mathematically
formulated as
\begin{equation}\label{6.3}
\F(\partial^m E)=\PP(E)
\end{equation}
for any bounded measurable subset of finite perimeter,
$E\subset\Omega$.

Fugele's theorem, Theorem \ref{ta1}, indicates that conditions
\eqref{6.1}--\eqref{6.2} implies that there is a production density
$P\in \M(\Omega; \R^k)$ such that
\begin{equation}\label{6.4}
\PP(E)=\int_{E^1} P(y).
\end{equation}

On the other hand, combining Theorem \ref{main} with the argument
from \S 9, it follows that there exists $\FF\in
\DM^\infty_{loc}(\Omega;\R^{N\times k})$ such that
\begin{equation}\label{6.5}
\F(\partial^m E)=-\int_{\partial^m E}(\mathscr{F}_i\cdot\nnu) \,
d\H^{N-1} =\int_{E^1} {\div}\, \FF(y)
\end{equation}
for any set of finite perimeter, $E\Subset\Omega$.

Then \eqref{6.3}--\eqref{6.5} yields the following system of field
equations
\begin{equation}\label{6.6}
\div\FF(y)= P(y)
\end{equation}
in the sense of measures on $\Omega$.

\medskip
We assume that the state of the medium is described by a state
vector field $u$, taking value in an open subset $U$ of $\R^k$,
which determines both the flux density field $\FF$ and the
production density field $P$ at the point $y\in\Omega$ by the {\it
constitutive equations}:
\begin{equation}\label{6.7}
\FF(y):=\FF(u(y),y), \quad P(y):=P(u(y),y),
\end{equation}
where $\FF(u,y)$ and $P(u,y)$ are given smooth functions defined on
$U\times \Omega$.

Combining \eqref{6.6} with \eqref{6.7} leads to the quasilinear
first-order system of partial differential equations
\begin{equation}\label{6.8}
\div \FF(u(y),y)=P(u(y),y),
\end{equation}
which is called a system of balance laws (cf. \cite{Da}).

If $\PP=0$, the previous derivation yields
\begin{equation}\label{6.9}
\div \FF(u(y),y)=0,
\end{equation}
which is called a {\it system of conservation laws}. When the medium
is homogeneous:
$$
\FF(u, y)=\FF(u),
$$
that is, $\FF$ depends on $y$ only through the state vector, then
system \eqref{6.9} becomes
\begin{equation}\label{6.10}
\div \FF(u(y))=0.
\end{equation}

In particular, when the coordinate system $y$ is described by the
time variable $t$ and the space variable $x=(x_1,\cdots, x_n)$:
$$
y=(t, x_1,\cdots, x_n)=(t,x),   \quad N=n+1,
$$
and the flux density is written as
$$
\FF(u)=(u, f_1(u), \cdots, f_n(u))=(u,f(u)),
$$
then we have the following standard form for the system of
conservation laws:
\begin{equation}\label{6.11}
\partial_t u + \nabla_x\cdot f(u)=0, \qquad x\in \R^n,\,\, u\in\R^k.
\end{equation}

\medskip
\section{Entropy Solutions of Hyperbolic Conservation
Laws}\label{sec:conservation}

We now apply the results established in \S 3 -- \S 9 to the
recovery of Cauchy entropy fluxes through the Lax entropy inequality
for entropy solutions of hyperbolic conservation laws by capturing
entropy dissipation. We focus on system \eqref{6.11} which is
assumed to be hyperbolic.

\begin{definition}\label{Definition4.1}
A function $\eta: \R^k\to \R$ is called an entropy of \eqref{6.11}
if there exists $q:\R^k\to \R^n$ such that
\begin{equation}
\nabla q_j(u)=\nabla \eta(u)\nabla f_j(u), \quad j=1,2,\dots, n.
\label{4.5}
\end{equation}
Then the vector function $q(u)$ is called an entropy flux associated
with the entropy $\eta(u)$, and the pair $(\eta(u),q(u))$ is called
an entropy pair. The entropy pair $(\eta(u), q(u))$ is called a
convex entropy pair on the domain $U\subset\R^k$ if the Hessian
matrix $\nabla^2\eta(u)\ge 0$ for any $u\in U$. The entropy pair
$(\eta(u),q(u))$ is called a strictly convex entropy pair on the
domain $U$ if $\nabla^2\eta(u)>0$ for any $u\in U$.
\end{definition}

Friedrichs-Lax \cite{FL} observed that most of systems of conservation
laws that result from continuum mechanics are endowed with a globally
defined, strictly convex entropy. The available existence theories
show that solutions of \eqref{6.11} are generally in the following
class of entropy solutions.

\begin{definition}\label{Definition2.1}
A vector function $u=u(t,x)\in L^\infty_{loc}(\R_+\times \R^n)$ is
called an entropy solution if $u(t,x)$ satisfies the Lax entropy
inequality:
\begin{equation}\label{entropy-ineq}
\po_t \eta(u(t,x))+ \nabla_x\cdot q(u(t,x)) \le 0
\end{equation}
in the sense of distributions for any convex entropy pair $(\eta,q):
\R^k\to \R\times \R^n$.
\end{definition}

Clearly, an entropy solution is a weak solution by choosing
$\eta(u)=\pm u$ in \eqref{entropy-ineq}.

One of the main issues in conservation laws is to study the behavior
of entropy solutions in this class to explore to the fullest extent
possible all questions relating to large-time behavior, uniqueness,
stability, structure, and traces of entropy solutions, with neither
specific reference to any particular method for constructing the
solutions nor additional regularity assumptions. The Schwartz lemma
infers from \eqref{entropy-ineq} that the distribution
$$
\po_t \eta(u(t,x)) +\nabla_x\cdot q(u(t,x))
$$
is in fact a Radon measure, that is, the field $(\eta(u(t,x)),
q(u(t,x)))$ is a divergence-measure field. Then there exists
$\mu_\eta\in \mathcal{M}(\R_+\times\R^n)$ with $\mu_\eta\le 0$ such
that
\begin{equation} \label{entropy-measure}
{\rm div}_{(t,x)}(\eta(u(t,x)),q(u(t,x)))=\mu_\eta.
\end{equation}

For any $L^\infty$ entropy solution $u$, it is first indicated in
Chen \cite{Chen1} that, if the system is endowed with a strictly
convex entropy, then, for any $C^2$ entropy pair $(\eta, q)$, there
exists $\mu_\eta\in \mathcal{M}(\R_+\times\R^n)$ such that
\begin{equation} \label{entropy-measure-1}
{\rm div}_{(t,x)}(\eta(u(t,x)),q(u(t,x)))=\mu_\eta.
\end{equation}

\medskip
We introduce a functional on any oriented surface $S$:
\begin{equation}\label{entropy-flux-1}
\F_\eta(S)=\int_S(\eta(u),q(u))\cdot\nnu \, d\H^{n},
\end{equation}
where $(\eta(u),q(u))\cdot\nnu$ is the normal trace in the sense of
Theorem \ref{main}, since $(\eta(u), q(u))\in
\DM^\infty_{loc}(\R_+\times\R^n)$. It is easy to check that the
functional $\F_\eta$ defined by \eqref{entropy-flux-1} is a Cauchy
flux in the sense of Definition \ref{def4}.

\begin{definition}[Cauchy Entropy Fluxes]
\label{Definition6.3}
A functional $\F_\eta$ defined by \eqref{entropy-flux-1} is called a
Cauchy entropy flux with respect to the entropy $\eta$.
\end{definition}
In particular, when $\eta$ is convex, then
$$
\F_\eta(S)\ge 0
$$
for any oriented surface $S$. Furthermore,
we can reformulate the balance law of entropy from the recovery of
an entropy production by capturing entropy dissipation.

On the other hand, it is clear that understanding more properties of
divergence-measure fields can advance our understanding of the
behavior of entropy solutions for hyperbolic conservation laws and
other related nonlinear equations by selecting appropriate entropy
pairs. As examples, we refer the reader to
\cite{CF1,CF2,CF-CMP,ChenWang} for the stability of Riemann
solutions, which may contain rarefaction waves, contact
discontinuities, and/or vacuum states, in the class of entropy
solutions of the Euler equations for gas dynamics; to
\cite{CF3,ChenWang} for the decay of periodic entropy solutions for
hyperbolic conservation laws; to \cite{CR,Vas} for the initial and
boundary layer problems for hyperbolic conservation laws; to
\cite{CF1,ChenTorres} for the initial-boundary value problems for
hyperbolic conservation laws; and to \cite{BFK,MPT} for nonlinear
degenerate parabolic-hyperbolic equations.

It is hoped that the theory of divergence-measure fields could be
used to develop techniques in entropy methods, measure-theoretic
analysis, partial differential equations, and related areas.

\bigskip
\noindent {\bf Acknowledgments.} The authors would like to thank
Luis Caffarelli, Constantine Dafermos, Willi J\"{a}ger, Fanghua Lin,
Leon Simon, and David Swanson for stimulating and fruitful
discussions. Gui-Qiang Chen's research was supported in part by the
National Science Foundation under Grants DMS-0505473, DMS-0244473,
and an Alexander von Humboldt Foundation Fellowship. Monica Torres's
research was supported in part by the National Science Foundation
under grant DMS-0540869.

\bigskip
\bigskip

\bibliographystyle{alpha}

\end{document}